\documentclass[11pt]{amsart}
\usepackage{geometry}                
\usepackage{graphicx}
\usepackage{amssymb}
\usepackage{epstopdf}
\usepackage{amsmath}
\usepackage{mathrsfs}
\usepackage[all]{xy}
\usepackage{color}
\usepackage[colorlinks=true, pdfstartview=FitV,linkcolor=blue,citecolor=blue,urlcolor=blue]{hyperref}

    \advance \topmargin by -\headheight
    \advance \topmargin by -\headsep

    \evensidemargin \oddsidemargin
\numberwithin{equation}{section}

\begin{document}

 \newtheorem{theorem}{Theorem}[section]
  \newtheorem{prop}[theorem]{Proposition}
  \newtheorem{cor}[theorem]{Corollary}
  \newtheorem{lemma}[theorem]{Lemma}
  \newtheorem{defn}[theorem]{Definition}
  \newtheorem{ex}[theorem]{Example}
    \newtheorem{conj}[theorem]{Conjecture}

 \newcommand{\cx}{{\bf C}}
\newcommand{\la}{\langle}
\newcommand{\ra}{\rangle}
\newcommand{\res}{{\rm Res}}
\newcommand{\expp}{{\rm exp}}
\newcommand{\Sp}{{\rm Sp}}
\newcommand{\lgp}{\widehat{\rm G}( { F} ((t)) )}
\newcommand{\svee}{\scriptsize \vee}
\newcommand{\deff}{\stackrel{\rm def}{=}}
\newcommand{\cal}{\mathcal}

\title{ LCA(2), Weil index, and Product Formula     }

\author[Dongwen Liu]{Dongwen Liu}
\address{School of Mathematical Science, Zhejiang University, Hangzhou 310027, Zhejiang, P.R. China}
\email{maliu@zju.edu.cn}

\author[Yongchang Zhu]{Yongchang Zhu}
\address{Department of Mathematics, The Hong Kong University of Science and Technology, Clear Water Bay, Kowloon, Hong Kong}
\email{mazhu@ust.hk}

\subjclass[2010]{Primary 43A25; Secondary 11R56, 42B35, 43A95}

\thanks{The second author is supported by Hong Kong RGC grant  16307214 and RGC grant 16305715}

\maketitle

\begin{abstract}
In this paper we study the category LCA(2) of certain non-locally compact abelian topological groups, and extend the notion of Weil index. As applications we deduce some product formulas for curves over local fields and arithmetic surfaces.
\end{abstract}

\section{ Introduction. }

The adeles of global fields are used in the formulation of class field theory and play a fundamental role in the
 study of automorphic forms.  The higher dimensional adeles were introduced by A. Parshin for surfaces \cite{P1} and were generalized later by A. Beilinson for schemes of higher dimensions \cite{B}.
 In the case that  $X$ is a surface over a finite field, its adele ring ${\Bbb A}_X $, as observed in \cite{OP1},
  is related to locally compact abelian groups (LCA for short) in the following way:  ${\Bbb A}_X $ is filtered by additive subgroups ${\Bbb A}(D)$ indexed by divisors $D$ of $X$,
  and for a pair  of divisors $D_1, D_2$ with $D_1 < D_2$,  ${\Bbb A} ( D_1 ) \subset {\Bbb A}  (D_2)$ and the quotient
   $ {\Bbb A} ( D_2 ) / {\Bbb A} ( D_1 )$ has a structure of  locally compact abelian group, as it is isomorphic to the adeles of some coherent sheaf in a one dimensional subscheme supported on $ | D_2 - D_1|$.
   The structure of such a system $( {\Bbb A}_X , \{ {\Bbb A}_D \} )$ was axiomatized by D. Osipov and A. Parshin in \cite{OP1, OP2}.
   The harmonic analysis for such spaces was also systematically developed in [{\it op. cit.}] where some ideas in \cite{K, K1} were used. Such a theory of harmonic analysis eventually leads to
   an analytic proof of Riemann-Roch theorem for algebraic surfaces \cite{OP3}. See also \cite{P2} for an excellent survey on this subject.

The present work will  generalize the results of  A. Weil  about locally compact abelian groups in \cite{W} to analytic objects in the category $C_2^{ar}$ used in \cite{OP2}. This category $C_2^{ar}$ is based on the categories $C_n$ introduced by
D. Osipov in \cite{O}, whose idea was to use filtered spaces and not to consider completions as in Tate (or
locally linearly compact) spaces. Thus
 an object in $ C_2^{ar} $ is roughly speaking a family of filtered abelian groups $\{ F(i), i\in I\}$ such that for any $ i , j \in I$, $i<j$,  $ F( i ) $ is embedded in $F ( j )$
 and the quotient $ F( j ) / F ( i)$ is endowed with a structure of locally compact abelian group.  An important list of examples are:
 adeles of a curve over a local field,  adeles of an algebraic surface over a finite field or an arithmetic surface, and some local versions in these examples.
 For simplicity,  we state the results in \cite{W} for the LCA group $G = F^n$, where $F$ is a local field (see Section \ref{s3} for the general case).
  Let $ q : G \to F$ be a non-degenerate quadratic form.  Then its composition with a non-trivial additive character $ \psi : F \to {\Bbb T}$
 gives a non-degenerate quadratic character $ h ( x ) = \psi ( q ( x ) ) $ of $G$, where ${\Bbb T}\subset \mathbb{C}^\times$ is the unit circle.   It was proved in \cite{W} that the Fourier transform of $ \psi ( q ( x ) ) $, as a tempered distribution on $G$,
  is of the form
  \[    {\cal F}  h ( x )  =   \gamma ( h )  \cdot  C \cdot     h' ( x ^* ) \]
  where  $\gamma ( h )\in{\Bbb T}$,    $ C>0$  and $ h' ( x^*)$ is a certain quadratic character of $G^*$ (see  Theorem \ref{windex} for more details).  The complex number  $ \gamma ( h)$ is nowadays called the Weil index of $h$.
  The same quadratic character $ h$ gives rise to an $SL_2 ( {\Bbb Z})$-action on the Heisenberg group associated with $G$ (see Theorem \ref{thm2.2.1}). It induces a projective representation
   of $SL_2 ( {\Bbb Z})$ on the Bruhat-Schwartz space ${\cal S} ( G)$, and the Weil index $\gamma ( h )$ appears in the central multipliers of the projective representation (see (\ref{2.3rel})).
  Now let $F$ be a global field, $q : F^n \to F$ be a non-degenerate quadratic form,  and $ \psi =\prod_v \psi_v $ be a non-trivial additive character of the adele ring ${\Bbb A}$ of $F$, which is trivial on $F$. For each place $v$,
   we have the Weil index $ \gamma ( h_v ) $ for the local quadratic character $ h_v ( x_v )  =  \psi_v ( q ( x_v ) ) $, and Weil's product formula (Theorem \ref{thm2.7}) states that
   \[   \prod_v \gamma ( h_v  )  = 1 . \]
   As is well known, Weil index and product formula are very important notions in representation theory and number theory, and these notions have made deep influences and wide applications on many subjects.

  In this work, we first introduce a category called LCA(2), whose objects are  abelian groups with compatible topology of two levels. Our definition emphasizes more on the
   analytic nature of these objects, and LCA(2) has more objects than $ C_2^{ar}$.  The results about this category are essentially due to Osipov and Parshin \cite{OP2}. With this being said, our approach makes improvements in certain aspects. For instance  the basis is no longer part of the datum of an object, which simplifies the theory significantly, and moreover our category LCA(2) has an exact category structure (Theorem \ref{exact}).
  The major new results of this work are generalizations of the results in \cite{W}  which we just recalled to the objects in the category LCA(2). Using the theory of Bruhat-Schwartz functions and Fourier transforms which follows \cite{OP2} closely, we are able to define the Weil index for a non-degenerate quadratic character of an LCA(2) object. As applications we deduce a few product formulas for curves over local fields, as well as arithmetic surfaces using a duality result in \cite{L}. The results on LCA(2) in this paper have been used in the theory of Weil representations and theta functionals on surfaces by the second author in \cite{Z}.

The content of this paper is as follows. In Section \ref{s3}, we recall the categorical properties of locally compact abelian groups and recall the results
in \cite{W} that we will generalize. In Section \ref{s4}, we introduce the category LCA(2), which is a modified version of the category $ C_2^{ar} $ first studied in \cite{OP2}. Our new definition
is ``base" free and is of more analytic nature.
   Though our formulation and that in \cite{OP2} are not logically equivalent,  all our results in Section \ref{s4} essentially appeared in \cite{OP2} already.
 In Section \ref{s5}, we introduce the notion of Weil index for a non-degenerate quadratic character on an LCA(2) group $G$. We prove that such a quadratic character gives rise to a projective representation
of $SL_2 ( {\Bbb Z} )$
  on the Bruhat-Schwarts space ${\cal S} ( G )$ of $G$ (Theorem \ref{thm4.5}).     In Section \ref{s6}, we prove the product formula of Weil index for curves over local fields and  local surfaces (Theorem \ref{5.2}, Theorem \ref{thm6.4}), and moreover
   we compute some concrete examples (Corollary \ref{cor5.3}, Corollary \ref{cor6.5}) explicitly.

\

\section{ Locally compact  Hausdorff topological abelian groups } \label{s3}

\subsection{Reminder of exact categories } \label{s2}

 In this subsection, we recall the notions related to exact categories following \cite{B1} and \cite{GM}.
An $Ab$-category is a category in which each hom-set $Mor ( B , C)$ is an additive abelian group and the composition
 of arrows is bilinear with respect to this addition.

Recall that  a biproduct diagram for two objects $A, B$ in an $Ab$-category  is a diagram
\[   A  { {  \stackrel { p_1}  { \longleftarrow }} \atop  { {\longrightarrow} \atop { i_1} } }   C  { {  \stackrel { p_2}  { \longrightarrow }}
 \atop  { {\longleftarrow} \atop { i_2} } }   B  \]
with arrows $p_1 , p_2 , i_1 , i_2 $ satisfying the identities
\[   p_1 i_1 = 1_A , \; \; \; p_2 i_2 = 1_B, \; \; \;  i_1 p_1 + i_2 p_2 = 1_C .\]
Then it is easy to see that
\[   A  {  \stackrel { p_1}  { \longleftarrow }}    C    \stackrel { p_2}  { \longrightarrow } B \]
is a product of $A$ and $B$, and
\[   A  {  \stackrel { i_1}  { \longrightarrow }}    C    \stackrel { i_2}  { \longleftarrow } B \]
is a coproduct of $A$ and $B$.

An  additive category is an $Ab$-category which has a zero object and a biproduct for each pair of its object.
A kernel-cokernel pair $(i, p )$ is a pair of morphisms
      \begin{displaymath}
\xymatrix{
  A_1  \ar[r]^{i}    &  A_2  \ar[r]^{p}
&  A_3  }
\end{displaymath}
such that $i$ is the kernel of $p$ and $p$ is the cokernel of $i$.  If a class $\Sigma $ of kernel-cokernel pairs is fixed, then an {\bf admissible monic}
 is a morphism $i$ such that there exists a morphism $p$ such that $(i, p )$ is in $\Sigma$, and similarly  an {\bf admissible epic}
  is a morphism $p$  such that there exists a morphism $i$ such that $(i, p )$ is in $\Sigma$.

\begin{defn} {\rm  An {\bf exact structure} on an additive category ${\cal C}$ is a class $\Sigma $ of kernel-cokernel pairs which is closed under isomorphisms and
   satisfies the following axioms:
\\
E1.  For all objects $ A \in  {\cal C}$, $1_A$ is an admissible monic and an admissible epic.
\\
 E2.  The class of admissible monics and the class of  admissible epics are closed under composition.
\\
E3.  The push out of an admissible monic along any morphism exists and yields an admissible monic.
\\
E4.   The pull back of an admissible epic along any morphism exists and yields an admissible epic.  }
\end{defn}

An additive category with an exact structure is called an {\bf exact category}.

\subsection{ LCA is an Exact category }

We denote by LCA the cateory of  locally compact Hausdorff topological abelian Groups.
For LCA groups $G_1$ and $G_2$,  i.e. objects in LCA, $Mor (  G_1 , G_2 )$ is the set of continuous group homomorphisms from $G_1$ to $G_2$. Clearly $Mor (  G_1 , G_2 )$
has an abelian group structure.  The biproduct for any two objects $G, H$ is the product group $G\times H$ with the obvious maps $i_1 , i_2, p_1 , p_2$.
So LCA is an additive category.  Any morphism $f: G \to H$ has a kernel  $ Ker ( f ) \to G $,  where $ Ker (f) = \{ a \in G \, | \, f ( a ) = 0 \}$
 is a subgroup of $G$ with the induced topology.    The cokernel of $f$  is
 $ H \to  H  /   \overline { Im( f ) } $ with the quotient topology, where $ \overline { Im( f ) }$ is the closure of the  $Im( f )$.

 Recall that in an additive category, if kernel and cokernel exist for any morphism, then
  pushout and pullback exist in this category. Therefore pushout and pullback exist  in   LCA.
 Consider a diagram in LCA:
         \begin{displaymath}
\xymatrix{
  G_1  \ar[r]^{i}    &  G_2  \ar[r]^{p}
&  G_3.}
\end{displaymath}
 It is a kernel-cokernel pair in LCA iff it is a short exact sequence of abelian groups, $i$ is an embedding of $G_1$ onto a closed subgroup $i (G_1)$ of $G_2$  and
  $p$ induces an isomorphism of the quotient group $G_2 / i ( G_1 ) $ onto $G_3$.
Let $ \Sigma $ be the class of all kernel-cokernel pairs in LCA.  It is clear that   $f:  G \to H$ is an admissible monic with respect to $\Sigma $
iff it is an closed embedding, and it is an admissible epic iff it induces an isomorphism
   $ G / Ker (f ) \to H $.

\begin{theorem}  The class of all kernel-cokernel pairs in LCA  is an exact structure.
\end{theorem}

This is a well-known result, see e.g., \cite{GM}. We sketch a proof.
It is clear that the axioms E1, E2 are satisfied,  and E3, E4 will be proved in the following lemma.

  \begin{lemma}\label{lemma2.2} Let
\begin{displaymath}
\xymatrix{
  G_1  \ar[r]^{i}    &  G_2  \ar[r]^{g}
&  G_3  }
\end{displaymath}
be an exact triple in LCA.
\newline (1) For any morphism $ f : B \to G_3$,  there exists  a diagram
\begin{displaymath}
\xymatrix{
  G_1  \ar[r]^{i'}  \ar[d]^=  &  G_2'  \ar[r]^{g'}  \ar[d]^{f'}
&  B  \ar[d]^f    \\
  G_1  \ar[r]^{i}  &  G_2  \ar[r]^{g}  &  G_3
}
\end{displaymath}
in LCA such that the right square is a pullback and the first row is also an exact triple.
\newline (2) For any morphism $ f : G_1 \to C$, there exists
 a diagram
\begin{displaymath}
\xymatrix{
  G_1  \ar[r]^{i}  \ar[d]^f  &  G_2  \ar[r]^{g}  \ar[d]^{f'}
&  G_3   \ar[d]^{=}    \\
  C  \ar[r]^{i'}  &  G_2'  \ar[r]^{g'}  &  G_3
}
\end{displaymath}
in LCA such that the left square is a pushout and the second row is also an exact triple.
\end{lemma}

\begin{proof} We prove (1). Let $G_2' =\{ ( a , b)\in G_2\times B \; | \;  g( a ) - f ( b ) = 0 \} $, which is a closed subgroup of $G_2 \times B$.
 Let  $f' : G_2' \to G_2$ and
 $g' : G_2' \to B$ denote the obvious projections. The embedding $i' : G_1 \to G_2 ' $ is given by
$a_1 \mapsto (a_1 , 0 )$. It is clear that the diagram in (1) is  a commutative diagram in LCA.
 By the general construction of pullback in an additive category with kernels and cokernels, we see the right
 square is a pullback. To prove  that $i' $ is a closed embedding, we consider the diagram
 \begin{displaymath}
\xymatrix{
  G_1  \ar[r]^{i'}   &  G_2'  \ar[r]^h  & G_2 \times  B
}
\end{displaymath}
where $h : G_2'\to G_2 \times B$ is the kernel of $ g \oplus (-f): G_2 \times B \to G_3 $, hence a closed embedding.
 The composition $h \circ i'$ is a closed embedding, therefore $i' : G_1 \to G_2'$ is a closed embedding as well. It is clear that
  $i'$ is a kernel of $g'$.   It remains to prove that   $\bar g :  G_2' / G_1 \to  B$ induced from $g'$ is an isomorphism in LCA. Obviously
$\bar g $ is an isomorphism of abelian groups. Hence it is enough to prove that $\bar g$ is a closed map, i.e. a map which sends closed sets to closed sets.  Since
the topology of $G_2' / G_1$ is the quotient topology,   it reduces to prove that
 $ g' : G_2' \to B$ is closed.  But $g'$ is the composition  $ G_2' \to G_2\times B \to B$, where the first map is closed (since it is an closed embedding) and the second map is the projection which is closed as well.
This proves (1), and (2) can be proved similarly.  \end{proof}

\subsection{Heisenberg group and automorphisms }\label{subsection2.2}

 In this subsection and the next, we recall the main results in \cite{W} that we will generalize to LCA(2).
 We will define the Heisenberg group $A(G)$ associated to an LCA $G$ and prove that a non-degenerate quadratic character $h(x)$
  on $G$ such that $h( x ) = h( -x )$ gives rises to an action of $SL_2 ( {\Bbb Z} ) $ on $A(G)$ as automorphisms.

Let $\Bbb T$ be the multiplicative group formed by complex numbers of norm $1$. Let $G$ be a locally compact abelian group, and $G^*=\textrm{Hom}_{\rm cts}(G, \Bbb T)$ be the Pontryagin dual
which consists of all the continuous homomorphisms from $G$ into $\Bbb T$.
  We shall write the operations
 in $G$ and $G^*$ additively, and write the
 pairing between $G$ and $G^*$ as $( x , x^* )$ for $x\in G$ and $x^*\in G^*$.
    We denote by $F$ the bicharacter on $(G \times G^*) \times (G \times G^*)$  given by
\begin{equation}\label{2.2.1}  F( ( x_1 , x_1^* ) , (x_2 , x_2^* ) ) = ( x_1 , x_2^* ) .\end{equation}
   The Heisenberg group $A(G)$ associated to $G$ is the set
 \[  A(G) = G \times G^* \times \Bbb T \]
 with the group law given by
 \begin{equation}\label{2.2.2} (w_1 , t_1 ) ( w_2 , t_2 ) = ( w_1 + w_2 , F(w_1 , w_2 ) t_1 t_2 ) .\end{equation}
 The center of $A(G)$ consists of the elements $(0,t)$, $t\in \Bbb T$, which can be identified with $\Bbb T$,
 and we can identify  the quotient group $A(G)/ \Bbb T$ with $G\times G^*$.

As in \cite{W}, let the automorphism
  group of $G \times G^*$ act on $ G \times G^*$ from the right. We often write an automorphism $\sigma $ as a $2\times 2$ matrix
 \begin{equation}\label{2.2.3}
      \sigma = \left( \begin{matrix}  \alpha & \beta \\ \gamma &\delta \end{matrix} \right),
 \end{equation}
which acts on $G\times G^*$ by
\[   (x , x^* )\sigma =  ( x, x^*  )\left( \begin{matrix}  \alpha & \beta  \\ \gamma &\delta  \end{matrix} \right)
     =  ( x \alpha + x^* \gamma ,  x \beta + x^* \delta ), \]
  where $\alpha , \beta , \gamma , \delta $ are homomorphisms from $G$ to $G$, $G$ to $G^*$, $G^*$ to $G$, and $G^*$ to $G^*$ respectively.
 We have a skew symmetric bicharacter  on  $G \times G^*$ given by
  \begin{eqnarray*}    \la  ( x_1 , x_1^* ) , (x_2 , x_2^* ) \ra   &=&  ( x_1 , x_2^* ) ( x_2 , x_1^*)^{-1}  \nonumber \\
&=& F( (x_1 , x_1^* ) , (x_2 , x_2^* )  ) F( (x_2 , x_2^* ) , (x_1 , x_1^* )  )^{-1}.  \end{eqnarray*}
   We call
an automorphism $\sigma$ of $G\times G^*$ a symplectic automorphism if
  it preserves the above  bicharacter, i.e.,
\[  F(w_1 \sigma , w_2 \sigma ) F(w_2\sigma , w_1 \sigma )^{-1} =  F(w_1 , w_2 ) F( w_2 , w_1 )^{-1},\]
  which is equivalent to
\begin{equation}\label{2.2.4}
    \alpha \delta^* - \beta \gamma^* = {1}_{G}, \; \; \alpha \beta^* = \beta \alpha^* , \; \; \gamma \delta^* = \delta \gamma^* .
\end{equation}
Denote by $Sp( G)$ the group of symplectic automorphisms of $G\times G^*$, and $B_0 ( G)$ the group of automorphisms of $A(G)$ that induce the identity map on the center $\Bbb T$.
  It is immediate that an automorphism $s$ of $A(G)$ can be expressed as
 \[ ( w , t ) s = ( w \sigma , f ( w ) t ) \]
 where $\sigma $ is an automorphism of $G\times G^*$ and $f: G\times G^*  \to  \Bbb T $ is a map.
 One checks that $( \sigma , f )$ is an automorphism iff
 \begin{equation}\label{2.2.5}
  f(w_1 + w_2 ) f(w_1 )^{-1} f(w_2 )^{-1} = F( w_1 \sigma , w_2 \sigma ) F(w_1 , w_2)^{-1} .\end{equation}
The group law of $B_0 ( G)$ is given by
 \[ ( \sigma ,f ) ( \sigma' , f' ) = ( \sigma \sigma' , f'')\]
where \[ f'' ( w ) := f(w) f' ( w \sigma ) .\]
Notice that the left hand side of (\ref{2.2.5}) is symmetric in $w_1 $ and $w_2$, so we have
 \[  F( w_1 \sigma , w_2 \sigma ) F(w_1 , w_2)^{-1} =  F( w_2 \sigma , w_1 \sigma ) F(w_2 , w_1)^{-1}. \]
 This implies that  $\sigma\in Sp( G )$,
  so we have a group homomorphism
\[   B_0 ( G ) \to Sp(G), \; \; \;   ( \sigma , f ) \mapsto \sigma ,\]
whose kernel is formed by elements of the form $( 1 , f )$. It is easy to see that
 \[ f ( u , u ^* ) = ( u , a^* ) ( a , u^* ) \]
for some $a\in G , a^* \in G^*$. Such an automorphism is exactly the inner automorphism given by
 $(a , a^*, 1 ) \in A(G)$.

If $x \mapsto 2x $ is an automorphism of $G$, we denote by $x \mapsto \frac 12 x $ the inverse automorphism.
In this case $B_0 ( G ) = Sp( G) \ltimes ( G \times G^* ) $. To prove this,  we use a different definition of the Heisenberg group $A( G )$.
 On $A(G) = G\times G^* \times \Bbb T $, we define a new product $*$ by
\[  ( w_1 , t _1 ) * (w_2 , t_2 ) =  ( w_1 + w_2 ,  F(\frac 12 w_1 , w_2 ) F(\frac 12 w_2 , w_1 )^{-1} t_1 t_2 ) .\]
It is easy to see that the map
 \begin{equation}\label{c2}  ( w , t ) \mapsto ( w ,  F(\frac 12  w , w )^{-1} t )   \end{equation}
is an isomorphism from the Heisenberg group $A(G)$ onto $(A(G), * )$.
   For each
 $  \sigma \in Sp(G)$,
\[   ( w , t ) \mapsto  (  w \sigma , t ) \]
 gives an automorphism of $(A( G ), *)$.
     Pulling back this automophism to $A(G)$ using (\ref{c2}), we  get
an automorphism of $A (G)$:
\[    ( w , t ) \mapsto ( w \sigma , F( \frac 12 w \sigma , w \sigma ) F( \frac 12 w , w )^{-1} t ). \]
 Therefore we obtain a lifting $  Sp( G) \to  B_0 ( G) $, $\sigma \mapsto (\sigma , F( \frac 12 w \sigma , w \sigma ) F( \frac 12 w , w )^{-1} )$.

We list some  special elements in  $B_0 ( G ) $. All the notations we use follow the original paper \cite{W}.
     If $\alpha : G \to G $ is an automorphism, then
   \begin{equation}\label{O1}
         d_0 ( \alpha ) := \left(  \left( \begin{matrix} \alpha  & 0 \\ 0 & \alpha^{*-1} \end{matrix} \right) , 1 \right) \end{equation}
is in $B_0 (G) $.
 If $\gamma : G^* \to G $ is an isomorphism of $G^*$ onto $G$, then
  \begin{equation}\label{O2} d_0' ( \gamma ) :=
 \left(  \left( \begin{matrix} 0  &  -\gamma^{*-1} \\ \gamma & 0  \end{matrix} \right) , f \right),
 \; \; \; {\rm where } \; \; \; f( x , x ^* ) = ( x , -x^* ) .\end{equation}
 If $h: G \to \Bbb T$ is a quadratic character, i.e.,  $h ( x_1 +x_2 ) h( x_1 )^{-1} h( x_2)^{-1}$ is a bicharacter on $G\times G$,
 let  $ \rho :  G \to G^*$ be given by
 \[  h ( x_1 +x_2 ) h( x_1 )^{-1} h( x_2)^{-1} = ( x_1 , x_2 \rho ). \]
Note that $\rho$ is symmetric, i.e., $\rho=\rho^*$. Then
 \begin{equation}\label{O3}  t_0 ( h ) :=  \left(  \left( \begin{matrix} 1  &  \rho  \\ 0  & 1  \end{matrix} \right) , h (x) \right)
 \end{equation}
is in $B_0 ( G)$. The map $h\mapsto t_0 ( h)$ is an embedding of the group $X_2 ( G)$ of
 quadratic characters of $G$ into $B_0 ( G ) $.
Similarly if $h': G^* \to \Bbb T$ is a quadratic character, then
  \begin{equation}\label{O4}  t'_0 ( h' ) :=  \left(  \left( \begin{matrix} 1  &  0  \\ \rho'  & 1  \end{matrix} \right) , h' (x^* ) \right)
   \end{equation}
and $h'\mapsto t_0' ( h')$ is an embedding of the group $X_2 ( G^* )$ of
 quadratic characters of $G^*$ into $B_0 ( G ) $.

\begin{theorem}\label{thm2.2.1}
 Let $h(x)$ be a quadratic character of $G$  satisfying $ h ( x ) = h( -x ) $,  so that \[
 h(x_1 + x_2 ) h (x_1 )^{-1} h( x_2 )^{-1} = ( x_1 , x_2 \rho ) \]
  for some homomorphism $\rho : G \to G^*$. Assume that $h$ is non-degenerate, i.e., $\rho$ is an isomorphism.
   Then there is a group homomorphism
 $SL_{2} ( {\Bbb Z} ) \to B_0 ( G ) $ given by
   \[   \left( \begin{matrix}   1 & 1 \\ 0 & 1 \end{matrix} \right) \mapsto t_0 ( h )  ,\quad  \left( \begin{matrix}   0 & -1 \\ 1 & 0 \end{matrix} \right) \mapsto d_0' ( \rho^{-1} ). \]
\end{theorem}

 This theorem was  implicitly contained in \cite{W}.

\begin{proof} Recall that the modular group $SL_{2} ( {\Bbb Z} )$ is isomorphic to the group with generators $s ,t $ and relations
 \begin{equation}\label{sl2}   s^4 = 1 , \; \; \; \; \;  (ts)^3 = s^2 . \end{equation}
 The isomorphism is given by
\[   s \mapsto  \left( \begin{matrix}   0 & -1 \\ 1 & 0 \end{matrix} \right),  \; \; t \mapsto  \left( \begin{matrix}   1 & 1 \\ 0 & 1 \end{matrix} \right) . \]
 Hence it is enough to prove that
  \begin{equation}\label{rel}      d_0' ( \rho^{-1} )^4 =1,   \;  \;  \; (  t_0 ( h )  d_0' ( \rho^{-1} ))^3 =   d_0' ( \rho^{-1} )^2.
                   \end{equation}
We have
  \begin{eqnarray*}
   d_0' ( \rho^{-1} )^2 &=&   \left(  \left( \begin{matrix} 0 &  -\rho  \\ \rho^{-1}  & 0  \end{matrix} \right) ,  ( x , -x^* ) \right)
                              \left(  \left( \begin{matrix} 0 &  -\rho  \\ \rho^{-1}  & 0  \end{matrix} \right) ,  ( x , -x^* ) \right) \\
                        &=& \left(  \left( \begin{matrix} -1 &  0  \\ 0  &  -1  \end{matrix} \right) , (x, -x^* ) ( x^* \rho , x \rho^{-1} )   \right)    \\
                        &=&    \left(  \left( \begin{matrix} -1 &  0  \\ 0  &  -1  \end{matrix} \right) , 1   \right),
                         \end{eqnarray*}
which implies that $  d_0' ( \rho^{-1} )^4 =1$.

  To prove the second relation in (\ref{rel}), we compute that
     \begin{align}\label{2c1}
   &  t_0 ( h ) d_0' ( \rho^{-1} ) t_0 ( h ) \\
    = &  \left(  \left( \begin{matrix} 1  &  \rho  \\ 0  & 1  \end{matrix} \right) , h (x) \right)
       \left(  \left( \begin{matrix} 0 &  -\rho  \\ \rho^{-1}  & 0  \end{matrix} \right) ,  ( x , -x^* ) \right)
        \left(  \left( \begin{matrix} 1  &  \rho  \\ 0  & 1  \end{matrix} \right) , h (x) \right)        \nonumber \\
    = &    \left(  \left( \begin{matrix} 1  &  \rho  \\ 0  & 1  \end{matrix} \right) , h (x) \right)
           \left(  \left( \begin{matrix} 0 &  -\rho  \\ \rho^{-1}  & 1  \end{matrix} \right) ,  ( x , -x^* ) h ( x^* \rho^{-1} ) \right)  \nonumber \\
   = &    \left(  \left( \begin{matrix} 1  &  0  \\ \rho^{-1}   & 1  \end{matrix} \right) , h (x)  (-x , x\rho + x^* ) h ( x +x^* \rho^{-1} ) \right) \nonumber \\
     = &    \left(  \left( \begin{matrix} 1  &  0  \\ \rho^{-1}   & 1  \end{matrix} \right) , h (x)^2 h(x^*\rho^{-1} ) (-x , x\rho  ) \right), \nonumber
     \end{align}
     \begin{align}
 \label{2c2}
   &  d_0' ( \rho^{-1} ) t_0 ( h )  d_0' ( \rho^{-1} ) \\
  =&   \left(  \left( \begin{matrix} 0 &  -\rho  \\ \rho^{-1}  & 0  \end{matrix} \right) ,  ( x , -x^* ) \right)
        \left(  \left( \begin{matrix} 1  &  \rho  \\ 0  & 1  \end{matrix} \right) , h (x) \right)
          \left(  \left( \begin{matrix} 0 &  -\rho  \\ \rho^{-1}  & 0  \end{matrix} \right) ,  ( x , -x^* ) \right)      \nonumber \\
    = &    \left(  \left( \begin{matrix} 0  &  -\rho  \\ \rho^{-1}  & 1  \end{matrix} \right) , (-x , x^* ) h (x^* \rho^{-1} ) \right)
          \left(  \left( \begin{matrix}  0  &  -\rho  \\ \rho^{-1}  & 0  \end{matrix} \right) ,  ( x , -x^* ) \right)  \nonumber \\
   = &    \left(  \left( \begin{matrix}  -1  &  0  \\ \rho^{-1}   & -1  \end{matrix} \right) , (-x , x^* ) h (x^* \rho^{-1} ) ( -x^* \rho^{-1},  -x \rho +x^*  ) \right) \nonumber \\
     = &    \left(  \left( \begin{matrix}  -1  &  0  \\ \rho^{-1}   & -1  \end{matrix} \right) ,  h (x^* \rho^{-1} ) ( -x^* \rho^{-1},  x^*  ) \right). \nonumber
   \end{align}
By (\ref{2c1}) and (\ref{2c2}), we have
  \begin{align*}
         &(  t_0 ( h )  d_0' ( \rho^{-1} ))^3  \\
         =&  \left(  \left( \begin{matrix} 1  &  0  \\ \rho^{-1}   & 1  \end{matrix} \right) , h (x)^2 h(x^*\rho^{-1} ) (-x , x\rho  ) \right)
        \left(  \left( \begin{matrix}  -1  &  0  \\ \rho^{-1}   & -1  \end{matrix} \right) ,  h (x^* \rho^{-1} ) ( -x^* \rho^{-1},  x^*  ) \right)       \\
=&  \left(  \left( \begin{matrix}  -1  &  0  \\ 0   &  -1  \end{matrix} \right) , h (x)^2 h(x^*\rho^{-1} ) (-x , x\rho  )  h(x^*\rho^{-1} ) ( -x^*\rho^{-1} ,x ^* )  \right)   \\
 =&  \left(  \left( \begin{matrix}  -1  &  0  \\ 0   &  -1  \end{matrix} \right) , \frac { h ( x) } { h ( -x) } \frac { h ( x^{-1} \rho^{-1}) } { h ( - x^{-1} \rho^{-1}) }  \right)   \\
 =&  \left(  \left( \begin{matrix}  -1  &  0  \\ 0   &  -1  \end{matrix} \right) , 1  \right)  \nonumber \\
=&   d_0' ( \rho^{-1} )^2 .
 \end{align*}
\end{proof}

Under the conditions of Theorem \ref{thm2.2.1}, we have a non-degenerate quadratic character $h'$ of $G^*$ defined by
\[
h'(x^*)=h(x^*\rho^{-1})^{-1}.
\]
It is straightforward to check that
\[  t'_0 ( h ' ) =  t_0 ( h )^{-1}    d_0' ( \rho^{-1} )^{-1}   t_0 ( h )^{-1} . \]
Hence the homomorphism in Theorem \ref{thm2.2.1} maps
 \[   \left( \begin{matrix}   1 & 0 \\ -1 & 1 \end{matrix} \right) \mapsto t'_0 ( h ' )  .\]

\

If $G$ is the $n$-dimensional vector space $F^n$ over a local field $F$, then $G^*$ is isomorphic to $G$ itself.
A non-degenerate quadratic form $ q : G \to F$ gives rise to a family of quadratic characters
$ h_c (  x ) := \psi ( q ( c x ) ) $ parametrized by $c \in F$.  Then
 \[   \left( \begin{matrix}   1 & c \\ 0 & 1 \end{matrix} \right) \mapsto t_0 ( h_c )  ,\quad  \left( \begin{matrix}   0 & -1 \\ 1 & 0 \end{matrix} \right) \mapsto d_0' ( \rho^{-1} ). \]
generate a homomorphism  $SL_2 ( F)\to B_0 ( G) $.

\subsection{Weil index of a non-degenerate quadratic character }\label{ss3.3}
Let $dx $ be a Haar measure on $G$ and $dx^*$ be the dual Haar measure on $G^*$.   The duality here means that
 the Fourier transform $F: L^2 ( G) \to L^2 ( G^*) $ given by
 \[   F f (x^*) = \int_G  f(x ) ( x , x^* ) d x   \]
 is unitary.
Let $h(x)$ be a non-degenerate quadratic character of $G$ as in Theorem \ref{thm2.2.1}. In this subsection we shall prove that
the Fourier transform of $h(x)$, as a tempered distribution on $G^*$, equals
   \[  \gamma ( h )   | \rho |^{- \frac 12 }  h ( x^* \rho^{-1} )^{-1}    \]
   for some  $\gamma ( h ) \in \Bbb T $ and $|\rho |$ the positive number defined by (\ref{2norm}) below.
    The number $\gamma ( h )$ is called the {\bf Weil index}  of $h$.  It is independent of the
  Haar measure $dx $ on $G$.

 The Hilbert space  $L^2 (G)$ is a unitary representations of the Heisenberg group $A(G)$ under the action
\[  \pi ( u , u^*, t ) f (x ) = t ( x , u^* ) f( x + u )   . \]
The space  ${\cal S} ( G )$, formed by the  Bruhat-Schwartz functions on $G$, is
  a dense subspace of  $L^2 (G)$ and is preserved by $A(G)$.  If we endow ${\cal S } ( G)$ with the usual locally convex topology,
 then the action map $ A(G) \times {\cal S} ( G ) \to {\cal S} ( G ) $ is continuous. Similarly
  $L^2 ( G^*)$ is a unitary representation of $A(G)$ under the action
\[    \pi ( u , u^*, t ) f (x^*  ) = t  (-u , x^* + u^* )   f( x^* + u^* ) .\]
The Fourier transform $F: L^2 ( G ) \to L^2 ( G^* ) $ is an isomorphism of unitary representations of $A(G)$, and  it maps  ${\cal S} ( G) $ to ${\cal S} ( G^* )$
    as an isomorphism of topological vector spaces.

 For the operators $ d_0 ( \alpha ) ),    d_0' ( \gamma )$ and  $ t_0 ( h) $ given by (\ref{O1}), (\ref{O2}) and (\ref{O3}) respectively,
    we define operators
   $\pi (  d_0 ( \alpha ) ) $, $ \pi ( t_0 ( h) ) $ and $ d_0' ( \gamma ) $ on ${\cal S} ( G ) $ by
\begin{align*}
&  \pi (  d_0 ( \alpha ) ) f ( x ) =  f ( x \alpha ) ,\\
& \pi ( t_0 ( h) ) f ( x ) = h ( x ) f ( x ) ,\\
&   \pi (  d_0' ( \gamma ) )   f ( x ) =          Ff  ( - x {\gamma^*}^{-1} ) .
\end{align*}
Let $ | \alpha |$ and $|\gamma | $ be the positive numbers defined by
\begin{equation}\label{2norm}
       \int_G  f ( x \alpha ) dx = \frac 1 {|\alpha | }  \int_G  f ( x  ) dx , \; \; \; \int_{G^*} f ( x^* \gamma )  d x^* =  \frac 1 { | \gamma |} \int_G f ( x ) d x.
\end{equation}
Then the operators $ |\alpha |^{1/2} \pi (  d_0 ( \alpha ) ) $, $\pi ( t_0 ( h ) )$ and  $| \gamma |^{-1/2}  \pi (  d_0' ( \gamma ) )$ preserve the $L^2$-norm and extend
 to unitary operators on $L^2 ( G ) $.

It is easy to check that if $ g$ is one of $ d_0 ( \alpha )$,  $t_0 ( h)$ or $ d_0' ( \gamma )$, then for any $a\in A(G)$,
 \begin{equation}\label{comm}    \pi (   a \cdot  g ) = \pi ( g)^{-1} \pi ( a ) \pi ( g ) .\end{equation}

Assume that $ h (x)$ is a non-degenerate quadratic character of $G$ and $ \rho : G \to G^* $ is as in Theorem \ref{thm2.2.1}.
The relations (\ref{rel}) and (\ref{comm}) imply that
 the operators
 \begin{equation}\label{scalars}
   \pi ( d_0' ( \rho^{-1} ))^4  ,      \quad  \left( \pi (  t_0 ( h ) )  \pi ( d_0' ( \rho^{-1} ))\right) ^3  \pi (  d_0' ( \rho^{-1} ))^{-2}
  \end{equation}
   commute with $\pi ( a )$ for all $ a \in A(G)$. Since $L^2 ( G)$ is an irreducible unitary representation of $A(G)$, the operators in (\ref{scalars})
  are scalars.
  Since $  |\rho |^{1/2 }  \pi ( d_0' ( \rho^{-1} )) $ and  $  \pi (  t_0 ( h ) )$ are unitary, we know the modulus of these scalars:
  \[   |  \pi ( d_0' ( \rho^{-1} ))^4 | =  |\rho |^{-2 }, \quad | \left( \pi (  t_0 ( h ) \pi ( d_0' ( \rho^{-1} ))\right)^3  \pi (  d_0' ( \rho^{-1} ))^{-2} | =  |\rho |^{-1/2 } .\]
  We compute that
 \begin{align}\label{S}
    & \pi ( d_0' ( \rho^{-1} ))^2 f ( x )  \\
    =& \int_{G}  Ff  ( -y \rho ) ( - x \rho , y ) dy   \nonumber \\
     =& |\rho |^{-1} \int_{G^*}    Ff  ( x^* ) ( - x \rho , -x^* \rho^{-1}  ) dx^*    \nonumber \\
      =& |\rho |^{-1} \int_{G^*}    Ff  ( x^* ) (  x , x^*) dx^*      \nonumber \\
     =& |\rho |^{-1} f( -x ),     \nonumber
 \end{align}
which implies that
\begin{equation} \label{4power}
\pi ( d_0' ( \rho^{-1} ))^4 = | \rho |^{-2}.
\end{equation}
We have
 \[ \left( \pi (  t_0 ( h ) )  \pi ( d_0' ( \rho^{-1} )) \right) ^3    = c  |\rho |^{-1/2 }  \pi (  d_0' ( \rho^{-1} ))^{2} \]
  for some $c\in \Bbb T$.  That is,
\begin{align} \label{2.2.10}
   & \pi (  t_0 ( h ) )  \pi ( d_0' ( \rho^{-1} ))  \pi (  t_0 ( h ) )    \\
=&  c  |\rho |^{-1/2 }   \pi (  d_0' ( \rho^{-1} ))^{2}    \pi ( d_0' ( \rho^{-1} ))^{-1}  \pi (  t_0 ( h ) )^{-1}   \pi ( d_0' ( \rho^{-1} ))^{-1} . \nonumber
    \end{align}
Note that
  $ \pi (  t_0 ( h ) )^{-1}  =  \pi  ( t_0 ( h^{-1}  ) ).$
  By (\ref{S}),
 \[   \pi ( d_0' ( \rho^{-1} ))^{-1}  f (x ) = |\rho|\pi(d_0'(\rho^{-1}))f(-x)=  |\rho |  \int_G f( y ) ( y , x\rho ) d y. \]
 Applying (\ref{2.2.10}) to $f(x)\in {\cal S} ( G ) $, we obtain
\begin{align*} \textrm{LHS of (\ref{2.2.10})}  &=  h ( x )  \int_G h( y )  f( y ) (- y , x\rho ) d y,  \\
\textrm{RHS of (\ref{2.2.10})} &= c |\rho |^{1/2} \int_G \left( ( t , - x\rho ) h(t)^{-1} \int_G  f( y ) ( y , t \rho ) dy \right) d t \\
    &= c |\rho |^{-1/2}\int_{G^*} \left( ( x^* , - x  ) h(x^*\rho^{-1} )^{-1} \int_G  f( y ) ( y , x^*  ) dy \right) d x^*.
  \end{align*}
Comparison of both sides evaluated at $x=0$ yields
\[  \int_G h ( y ) f ( y ) dy =  c |\rho |^{-1/2} \int_{G^*}  h(x^*\rho^{-1} )^{-1}   Ff  ( x^* ) d x^*. \]
 Replacing $ f$ by  $Fg $ where $g\in \cal S(G^*)$, we get
  \[  \int_G h ( y )  F g  ( y ) d y   =  c |\rho |^{-1/2} \int_{G^*}  h(x^*\rho^{-1} )^{-1}  g  ( - x^* ) d x^*
                                       = c |\rho |^{-1/2} \int_{G^*}  h(x^*\rho^{-1} )^{-1}  g  (  x^* ) d x^* .\]
 This holds for all $g\in {\cal S} ( G^* )$, so we have  the identity of distributions
\begin{equation}\label{weilindex}
 F h ( x^* ) =   c   |\rho |^{- 1/2}    h (x^* \rho^{-1} )^{-1}.
 \end{equation}
We call $c$ the Weil index of $h$ and denote it by $ \gamma ( h ) $. In summary, we have following (\cite[Theorem 2]{W}):

\begin{theorem}\label{windex} Let $h(x)$ be a non-degenerate quadratic character of $G$ as in Theorem \ref{thm2.2.1}. Then
 the Fourier transform of $h(x ) $, as a distribution on $G^*$, is given by
    (\ref{weilindex}).
\end{theorem}

By (\ref{4power}) and
\begin{equation}\label{2.3rel}
   \left( \pi (  t_0 ( h ) )  \pi ( d_0' ( \rho^{-1} )) \right) ^3    = \gamma ( h)   |\rho |^{-1/2 }  \pi (  d_0' ( \rho^{-1} ))^{2} ,
\end{equation}
 $ \pi ( d_0' ( \rho^{-1} )) $ and $ \pi ( t_0 ( h) )  )$ satisfy the relation (\ref{sl2}) up to a scalar,  so they generate an action of
a central extension of $SL_2 ( {\Bbb Z} ) $ on $L^2 ( G ) $.

\begin{prop} \label{prop2.5}  Let $h$ be a non-degenerate quadratic character of $G$ satisfying $h( x ) = h( -x )$ and $\rho : G\to G^*$ be the associated isomorphism.
  If a closed subgroup $U \subset G$ satisfies the conditions that $ U^{\perp } \rho^{-1} \subset U$ and $h(U^\perp \rho^{-1})=1$,
  define an induced  quadratic character $\bar h $ of $ \bar U := U /   U^{\perp } \rho^{-1} $ by
  \[  \bar h ( \bar x ) = h ( x )   \]
  where $ x\in U  $ is any lift of $\bar x \in \bar U $. Then $ \bar h $ is well-defined
   and non-degenerate,  and $\gamma (  h ) = \gamma (\bar h ) $, i.e., the Weil index of $ h$ equals the Weil index of $\bar h$.
  \end{prop}

\begin{proof} It is clear that $\bar h$ is well-defined and non-degenerate.
  We see that the dual $\bar U^*$  of $\bar U$ is $ U \rho / U^{\perp} $.
  The map $\bar \rho :   \bar U \to  \bar U^*$ given by $ \bar h( \bar{x} + \bar{y} ) = \bar h (\bar{x} ) \bar h ( \bar{y} ) ( \bar{x} , \bar{y}\bar \rho )$ is
  the map induced from $\rho $:
  \[    \bar U = U / U^{\perp } \rho^{-1} \to \bar U^* =   U \rho / U^{\perp}  .\]
  We have the operators
    \[  \pi ( t ( \bar h ) ) , \quad \pi ( d_0' ( {\bar \rho}^{-1} )): \; \; {\cal S} ( \bar U ) \to {\cal S} ( \bar U ), \]
  which satisfy relations similar to (\ref{4power}) and (\ref{2.3rel}).
   Fix Haar measures on  $U^{\perp } \rho^{-1}$ and $U$, which then induce  Haar measures on $\bar U$ and $\bar G:= G/U^\perp\rho^{-1}$. Note that $\bar G =(G^*/U^\perp)\rho^{-1} =U^*\rho^{-1}$. Consider the natural projection $G\stackrel{p}{\longrightarrow} \bar G$ and closed embedding $\bar U\stackrel{ i}{\longrightarrow}\bar G$.
    Define a surjective map
    $ I:   {\cal S} ( G ) \to {\cal S} ( \bar U ) $ as the composition
           \begin{displaymath}
\xymatrix{
\cal S (G)  \ar[r]^{p_*}    &  \cal S(\bar G)  \ar[r]^{i^*}
&  \cal S(\bar U) }
\end{displaymath}
where $p_*$ is the integration over the fibre of $p$,
\[    p_*f ( \bar{x} ) =    \int_{ U^{\perp } \rho^{-1} }      f ( x + t )   dt,\quad f\in \cal S(G), \
p(x)=\bar x,
\]
and $i^*$ is the restriction.  Consider the following diagrams
\[
\xymatrix{
     {\cal S}  ( G ) \ar[r]^{ \pi ( t_0 ( h )) }    \ar[d]^{I}   &   {\cal S} ( G ) \ar[d]^I \\
    {\cal S} ( \bar U  )   \ar[r]^{ \pi ( t_0 (\bar  h )) }   &  {\cal S} ( \bar U   )
}    \; \; \; \; \; \; \ \
\xymatrix{
     {\cal S} ( G ) \ar[r]^{ \pi ( d_0' ( \rho^{-1} )  ) }  \ar[d]^I  &   {\cal S}( G ) \ar[d]^I  \\
    {\cal S} ( \bar U  )   \ar[r]^{ \pi ( d_0' ( {\bar \rho}^{-1} )   ) }  &  {\cal S} ( \bar U   )
}
\]
It is clear that the left diagram commutes, and we shall prove that the right diagram commutes up to a positive scalar depending on the
 the choices of Haar measures on $G$ and $\bar U$ used in the definition of $ \pi ( d_0' ( \rho^{-1} )  )$ and  $\pi ( d_0' ( {\bar \rho}^{-1} )   )$, which then by
   (\ref{2.3rel}) implies that $ \gamma ( h ) = \gamma ( \bar h ) $.

To this end, take $f\in \cal S(G)$ and compare the functions in $\cal S(\bar U)$,
   \begin{align*}
  & f_1(\bar{x})= I\circ \pi(d_0'(\rho^{-1}))f(\bar{x})=\int_{U^\perp\rho^{-1}}Ff(-(x+t)\rho)dt,\\
  &
 f_2(\bar{x})=  \pi(d_0'(\bar{\rho}^{-1})) \circ I f(\bar{x})=\int_{\bar U} \int_{U^\perp\rho^{-1}} f(y+t)dt (\bar{y},-\bar{x}\bar{\rho}) d\bar{y}=\int_U f(y)(y,-x\rho)dy.
   \end{align*}
  It is more convenient to view $f_1,$ $f_2$ as functions in $\cal S(\bar G)$. In fact, we may extend them as $f_1(\bar x)=p_*(Ff(-x\rho))\in \cal S(\bar G)$ and $f_2(\bar x)=F_Uf(-x\rho)\in \cal S(U^*\rho^{-1})=\cal S(\bar G)$, where $F_U: \cal S(U)\to \cal S(U^*)$ is the Fourier transform over $U$. Then it suffices to show that $F_{\bar G} f_1$ is a positive multiple of $F_{\bar G}f_2$, where $F_{\bar G}: \cal S(\bar {G} )\to \cal S(\bar G^*)$ is the Fourier transform over $\bar G$. Take $x^*\in \bar G^*=U\rho$. We compute that
  \begin{align*}
  F_{\bar G}f_1( x^*)&=\int_{\bar G}\int_{U^\perp\rho^{-1}}Ff(-(x+t)\rho)dt( \bar x, \bar x^*)d\bar x\\
  &=\int_G Ff(-x\rho)(x, x^*) dx\\
  &=\frac{1}{|\rho|}\int_{G^*} Ff(y^*)(y^*, -x^*\rho^{-1})dx\\
  &=\frac{1}{|\rho|} f(x^*\rho^{-1}),\\
  F_{\bar G}f_2(x^*)&=\int_{U^*\rho^{-1}}F_Uf(-x\rho)(x, x^*)dx\\
  &= \frac{1}{|\rho_{\bar G}|}\int_{U^*} F_Uf(y^*)(y^*, -x^*\rho^{-1})dx\\
  &=\frac{1}{|\rho_{\bar G}|}f(x^*\rho^{-1})
  \end{align*}
    by the Fourier inversion formula, where $\rho_{\bar G}: \bar{G}=U^*\rho^{-1}\to U^*$ is induced from $\rho$. This finishes the proof of the proposition.
\end{proof}

\begin{cor}\label{cor2.6}  Let $h$ be a non-degenerate quadratic character of $G$ satisfying $h( x ) = h( -x )$ and $\rho : G\to G^*$ be the associated isomorphism. If there is  a closed subgroup $U \subset G $ satisfying $ U^{\perp} \rho^{-1} = U$ and $h(U^\perp \rho^{-1})=h(U)=1$, then
   $ \gamma ( h ) =1 $.
\end{cor}

Corollary \ref{cor2.6} is \cite[Theorem 5]{W}. It was used in \cite{W} to prove the following product formula for Weil index of a non-degenerate
 quadratic from over a global field.

\begin{theorem}\label{thm2.7}  Let $k$ be a global field, $M$ be a finite-dimensional vector space over $k$,
 $q $ be a non-degenerate quadratic form on $M$.  Let $ \psi = \prod_v \psi_v$ be a nontrivial additive character of the adele ring ${\Bbb A}_k$, which is trivial
  restricted on $k$.  Then $ \psi_v ( q ( x ) ) $ is a non-degenerate quadratic character of $M_v =  M\otimes_k k_v$. Let $\gamma_v ( q )$ be
  the Weil index of $ \psi_v ( q ( x ) )$. Then $\gamma_v ( q )=1$ for almost all $v$ and
  \[   \prod_v \gamma_v ( q ) = 1 .\]
\end{theorem}

\begin{proof}
 It is  \cite[Proposition 5]{W}, and we shall sketch a proof.
 Consider the quadratic character on the adelic space $M_{\Bbb A} := M\otimes_k {\Bbb A}_k = \prod_v' M_v $.
 Then $q$ and $\psi $ define a non-degenerate quadratic character
  \[  \psi_q : M_{\Bbb A}  \to {\Bbb T} , \ \  \{ x_v \}_v \mapsto \prod_v  \psi_v ( q ( x_v ) ) . \]
  Using the fact that $  a= ( a_v ) \in {\Bbb A}_k $ is in $k$ iff and
  $ \psi ( a r ) = 1 $ for all $ r\in k$,
 we can prove  that $\psi_q$ induces an isomorphism $\rho: M\stackrel{\sim}{\rightarrow}  M^\perp$. By Corollary \ref{cor2.6}, we have  $\gamma ( \psi_q )=1$. Fix a linear isomorphism $M\cong k^n$ by taking a basis, where $n=\dim M$, and let $q$ be represented by a matrix $Q\in GL_n(k)$. Then there exists a finite set $S$ of places of $k$ including all the archimedean ones,  such that
  $\psi_v$ has conductor $\cal O_v$ and $Q\in GL_n(\cal O_{k_v})$ for $v\not\in S$. It is easy to check that $\psi_v(q(\cdot))$ induces an isomorphism $\rho_v: \cal O_v^n\stackrel{\sim}{\rightarrow} (\cal O_v^n)^\perp$ for any such $v$. By
  Corollary \ref{cor2.6} again, we have $\gamma_v(q)=1$ for $v\not\in S$.

It is clear that if $q_i$ is a non-degenerate quadratic character of an LCA group $G_i$, $i=1,2$, then $\gamma(q_1\times q_2)=\gamma(q_1)\gamma(q_2)$. Applying this observation to
$
M_S=\prod_{v\in S}M_v$ and $M^S=\prod'_{v\not\in S}M_v,
$
where $S$ is as above,
it follows that
\[
\gamma(\psi_q)=\prod_v\gamma_v(q).
\]
This finishes the proof.
\end{proof}

Define  a quadratic character $ h' ( x^* ) = h ( x^* \rho^{-1} )^{-1} $ of $G^*$ under the conditions in Theorem \ref{thm2.2.1}. Let $\pi ( t_0' ( h' ) )$ be given by
\[       \pi ( t_0' ( h' ) ) =  \pi ( t_0 ( h ))^{-1}   \pi (    d_0' ( \rho^{-1} ))^{-1} \pi (   t_0 ( h ))^{-1}  . \]
 Then
 \[    \pi ( t_0' ( h' ) ) \phi ( x ) = |\rho |  h( x)^{-1}  \int_G   (-x\rho  , y )   h (  y )^{-1} \phi ( y ) dy   = |\rho |  \int_G   h ( x - y )^{-1} \phi ( y ) dy
    .               \]

\subsection{Warm-up for LCA(2)} Before introducing the general definition of  LCA(2) groups and product formula of Weil index in the later sections, we first give an example.

 Let $V$ be an infinite-dimensional vector space over a local field $k$, and $U_0$ be a given subspace. Let $h$ be a non-degenerate quadratic character of $V$, and $\rho: V\to V^\vee=\textrm{Hom}_k(V,k)$ be the induced inclusion. If there is a subspace $U$ of $V$ commensurable with $U_0$ such that $U^\perp\rho^{-1}\subset U$ and $U/U^\perp \rho^{-1}$ is finite-dimensional, then $h$ induces a non-degenerate quadratic character $\bar{h}$ on $U/U^\perp \rho^{-1}$. We define the Weil index $\gamma(h)$ of $h$ to be $\gamma(\bar{h})$, which only depends on the commensurable class of $U_0$ by Proposition \ref{prop2.5}.  Indeed, take another subspace $U'$ commensurable with $U_0$ such that $U'^\perp \rho^{-1}\subset U'$. Replacing $U'$ by $U+U'$ if necessary, by symmetry we may assume that $U\subset U'$. Then obviously
 $U'^\perp\rho^{-1}\subset U^\perp\rho^{-1}\subset U\subset U'$. Also $U'/U'^\perp\rho^{-1}$ is finite-dimensional since $\dim U'/U=\dim U^\perp\rho^{-1}/U'^\perp\rho^{-1}$ is finite. Let $\bar{h}'$ be the induced quadratic character on $U'/U'^\perp\rho^{-1}$. By Proposition \ref{prop2.5} we have $\gamma(\bar{h})=\gamma(\bar{h}')$ hence $\gamma(h)$ is well-defined.

For example let $V=k((t))^n$ and $U_0=k[[t]]^n$. Assume that $h(x)=\psi(\textrm{Res}(xQx^T \omega ))$ for some $Q\in GL_n(k((t)))$ and nonzero $\omega\in k((t))dt$,
where $\psi$ is a nontrivial additive character of $k$. Then the above conditions are satisfied
and we may define the Weil index $\gamma(h)$.

Now assume that $k$ is global, and $\psi=\prod_v\psi_v$ is a character of the adele ring $\mathbb{A}_k$ of $k$, which is trivial restricted on $k$. Assume that the datum $Q$ and $\omega$ are rational, and denote by $\gamma_v(h)$ the Weil index of the quadratic character $\psi_v(\textrm{Res}(xQx^T \omega ))$ of $V_v:=k_v((t))^n$, which depends on $Q$ and $\omega$. Then $\gamma_v(h)=1$ for almost all $v$ and one has the product formula
\[
\prod_v \gamma_v(h)=1.
\]

\section{ Category LCA(2)} \label{s4}

In this section we introduce the category LCA(2) and prove that it is an exact category, and we  also introduce the space of Bruhat-Schwartz functions on an LCA(2) group.

\subsection{\bf Category LCA(2) }\label{ss4.1}

An object in LCA(2) is an abelian group $G$ with a nonempty class ${\cal U} ( G )$ of abelian subgroups of $G$ such that
 for any $ S_1 \subset S_2 $ in ${\cal U } ( G ) $, the quotient group $S_2 / S_1 $ has an LCA structure, and the following axioms are satisfied.
\\ LCA(2)-1.  $ \bigcup_{  S\in {\cal U} ( G )  }  S = G $, $ \bigcap_{  S\in {\cal U} ( G )  }  S = \{ 0 \}  $.
\\   LCA(2)-2.  For any $ S_1 , S_2 \in {\cal U} ( G ) $, there exist $S , S' \in {\cal U} ( G ) $ such that
 $ S \subset S_i \subset S'$, $i=1,2$.
\\LCA(2)-3.   For $S_1 ,S_2 , S_3  \in {\cal U } ( G )$ with $S_1 \subset S_2 \subset S_3$, the sequence
 \[    S_2 / S_1 \to S_3 /S_1 \to  S_3 / S_2  \]
 is an exact triple in LCA.
\\LCA(2)-4.    For $ S_1 , S_2 \in {\cal U } ( G ) $ with $S_1 \subset S_2$, if $ U \subset S_2 / S_1 $ is a closed subgroup,
 then $\pi^{-1} ( U ) $, the preimage of $U$ under the natural map $ \pi : S_2 \to S_2 / S_1 $, is in ${\cal U } ( G )$.

An object in LCA(2) is called an LCA(2) group. We often denote ${\cal U} ( G )$ by ${\cal U}$. A subfamily
 ${\cal B} \subset {\cal U} $ is called a {\bf basis} of $( G , {\cal U})$ if for every $ S\in {\cal U}$, there exist
 $ S_1 , S_2 \in {\cal B}$ such that $   S_1 \subset S \subset S_2 $.

  To construct examples of objects in  LCA(2), we use the following lemma.

\begin{lemma}\label{lemma3.1}   Let $G$ be an abelian group with a nonempty class of subgroups ${\cal B}  $ such that for every $S_1 , S_2 \in {\cal B} $
 with $S_1 \subset S_2$, the quotient $S_2 /S_1 $ has an LCA structure. Assume that  LCA(2)-1,  LCA(2)-2 and  LCA(2)-3 hold for ${\cal B}$, then
 there is a unique LCA(2) structure ${\cal U}$ on $A$ such that ${\cal B}$  is a basis of $(G , {\cal U}) $, and that for   $S_1 , S_2 \in {\cal B}$
with $S_1 \subset S_2$ , the LCA structures on $S_2 / S_1$ for ${\cal B}$ and ${\cal U}$ coincide.
     \end{lemma}

\begin{proof} (sketch)   For $S_1 ,  S_2 \in {\cal B}$ with $S_1 \subset S_2$, we let
 $[ S_1 , S_2 ]$ be the set of subgroups $H$ satisfying $ S_1 \subset H \subset S_2$ and $H/ S_1 $ is a closed subgroup of $S_2/ S_1$.
 It is easy to prove that for $S_1 , S_2 , S_1' , S_2' \in {\cal B}$ with $S_1' \subset  S_1 \subset S_2 \subset S_2' $, and a subgroup $H$ with
 $S_1 \subset H \subset  S_2 $,   we have $ H \in [S_1 , S_2 ] $ iff  $ H \in [S_1' , S_2' ] $.
 Let $ {\cal U}$ be the union of all $[S_1 , S_2 ]$ for $S_1\subset  S_2$  in ${\cal B}$.   For $H_1 , H_2 \in {\cal U}$ with $H_1 \subset H_2$, since { LCA(2)-2} holds for
  ${\cal B}$, we can find $S_1 , S_2 \in {\cal B}$ such that $ S_1 \subset H_1  \subset H_2 \subset S_2$. Then we have closed subgroups
   $  H_1/  S_1 \subset H_2 /S_1 \subset  S_2 / S_1$, and we give an LCA structure on $H_2 / H_1 $ as the quotient LCA of $H_2/S_1$ by $H_1/S_1$. It is easy to see that
    the LCA structure on $H_2 / H_1 $ is independent of the choice of $ S_1 , S_2$.
   \end{proof}

\noindent {\bf Example 3.1.1}.        Let $G$ be an LCA group and $G((t))$ be the set of Laurent power series in $t$ with coefficients in $G$.
   Then $G((t))$ has an obvious abelian group structure.  Let $ {\cal B} = \{  t^n G[[t]] \, | \, n \in {\Bbb Z} \}$.
   For any $n \geq m $,   $ t^n G [[ t]] \subset   t^m G [[ t]]$ and we give the quotient $   t^m G [[ t]] /  t^n G [[ t]] \cong G^{n -m }$
  the obvious LCA structure.  It is clear that LCA(2)-1, LCA(2)-2 and LCA(2)-3 hold for ${\cal B}$. By Lemma \ref{lemma3.1},
   there is
   a unique LCA(2) structure on  $G((t))$ with ${\cal B}$ as a basis.
       In particular, if $k$ is a local field, then $ k(( t))^n = k^n (( t))   $ has an  LCA(2) structure.

\noindent {\bf Example 3.1.2}.   Let    $V$ be an infinite dimensional vector space over a local field $k$, and $W_0 $ be a subspace.
 Let $ {\cal B} $ be the set of  subspaces $W$ commensurable with $W_0$.   For any $W_1 , W_2 \in {\cal B} $ with $W_1 \subset W_2 $,
  since $W_1 , W_2$ are commensurable, the quotient $W_2/W_1$  is a finite-dimensional vector space over $k$, hence in particular an LCA group.
   All the conditions    LCA(2)-1,  LCA(2)-2 and  LCA(2)-3 hold for ${\cal B}$.  This basis gives $V$ an  LCA(2) structure
    that only depends on the commensurable class of $W_0$.

We remark that every object in the category  $C^{ar}_2$  defined in  \cite{OP2} has an LCA(2) structure.

Let $G, G'$ be objects in LCA(2). A morphism from $G$ to $G'$ is a map $ f: G \to G'$  satisfying
\\ Mor1.  $f$ is a group homomorphism.
\\Mor2.  For every $  S' \in {\cal U} ( G' ) $, there is $S\in  {\cal U} ( G )$ such that $ f ( S ) \subset S'$.
\\  Mor3.  For every $  S \in {\cal U} ( G ) $, there is $S'\in  {\cal U} ( G' )$ such that $ f ( S ) \subset S'$.
\\  Mor4.  For $S_1 , S_2 \in  {\cal U} ( G )$ with $S_1 \subset S_2 $ and for $S_1' , S_2' \in  {\cal U} ( G' )$ with $S_1' \subset S_2' $,
 if $f ( S_1 ) \subset S_1' $ and $f(S_2 ) \subset S_2 '$, then the induced map
   \[    S_2 / S_1 \to   S_2' / S_1' \]
 is a morphism in LCA, i.e., a continuous homomorphism of topological groups.

We first need to show that a composition of two morphisms is again a morphism.

\begin{lemma}\label{lemma3.2}
If $f: G_1\to G_2$ and $g: G_2\to G_3$ are morphisms in LCA(2), then $h=g\circ f: G_1\to G_3$
is also a morphism in LCA(2).
\end{lemma}

\begin{proof} It is easy to check that $h$ satisfies the first three axioms of morphisms. Let us prove that  Mor4 also holds for $h$.
Assume that $S_i, T_i\in {\cal U}(G_i)$, $S_i\subset T_i$, $i=1,3$, and $h(S_1)\subset S_3, h(T_1)\subset T_3$. We need to show
the induced map $\bar{h}: T_1/S_1\to T_3/T_3$ is a morphism in LCA. To this end it suffices to find $S_2, T_2\in {\cal U}(G_2)$ such that we have
$S_2\subset T_2$ and  $S_1\stackrel{f}{\to} S_2 \stackrel{g}{\to}S_3$, $T_1\stackrel{f}{\to} T_2\stackrel{g}{\to}T_3$.

Choose $T'_2\in{\cal U}(G_2)$, $T_3'\in{\cal U}(G_3)$ such that $f(T_1)\subset T_2'$, $T_3\subset T_3'$ and $g(T_2')\subset T_3'$.
Choose $S_2'\in{\cal U}(G_2)$ such that $S_2'\subset T_2'$ and $g(S_2')\subset S_3$. Consider the following diagram of maps
 \begin{displaymath}
\xymatrix{
T_2' \ar[r]^{\pi} & T_2'/S_2' \ar[r]^{g_1} \ar[dr]^{g_2} & T_3'/S_3 \ar[d]\\
& & T_3'/T_3
}
\end{displaymath}
Then $g_1$ and $g_2$ are morphisms in LCA. Let $S_2=\pi^{-1}(\mathrm{Ker}(g_1))=T_2'\cap g^{-1}(S_3)$ and $T_2=\pi^{-1}(\mathrm{Ker}(g_2))=T_2'\cap g^{-1}(T_3)$. Then $S_2, T_2\in{\cal U}(G_2)$ and satisfy our requirements.
\end{proof}

Thus we have defined the category LCA(2).   It is more convenient to use the following lemma to verify that a given map is a morphism

\begin{lemma}\label{lemma3.3}   Let $(G, {\cal U} ) , (G', {\cal U}')$ be objects in LCA(2).
 Let ${\cal P} \subset {\cal U}\times {\cal U}$ be a set with the properties that $ (S , T)\in {\cal P}$ implies $S\subset T$ and that
for any finitely many $S_1 , \dots , S_n \in {\cal U}$, there is $(S, T) \in {\cal P}$ such that
   $  S \subset S_i \subset T$ for all $S_i$. Suppose ${\cal P}'\subset {\cal U}'\times {\cal U}'$ satisfies the similar properties.
    Let ${\cal B}$  (resp. ${\cal B}'$) be a basis of $G$ (resp. $G'$).
         Assume that a group homomorphism $f : G \to G'$ satisfies the following
  \newline (1)  For every $  S' \in {\cal B}'  $, there is $S\in  {\cal B} $ such that $ f ( S ) \subset S'$.
\newline (2) For every $  S \in {\cal B}  $, there is $S'\in  {\cal B}' $ such that $ f ( S ) \subset S'$.
\newline (3)  For $(S , T) \in  {\cal P}$ and $(S' , T') \in  {\cal P}' $ with
  $f ( S ) \subset S' $ and $f( T ) \subset T'$, the induced map
   \[    T / S \to   T' / S' \]
 is a morphism in LCA.
 \newline Then $f$ is a morphism in LCA(2).
\end{lemma}

\begin{proof}  It is clear that $f$ satisfies  Mor1,  Mor2 and  Mor3.
    For $S_1 , S_2 , S_1' , S_2'$ as in  Mor4,   we can find
      $ ( S , T) \in {\cal P}$ and $ (S' , T' ) \in {\cal P}'$ such that
    \[  S \subset  S_1 \subset S_2 \subset T , \; \; \; \;   S' \subset  S_1' \subset S_2' \subset T' , \; \; \;  f ( S) \subset S' ,  f ( T) \subset T' .\]
    Consider the following commutative diagram in the category of abelian groups for $ i = 1, 2$,
\begin{displaymath}
\xymatrix{
  S_i / S   \ar[r]  \ar[d]  &  T / S     \ar[d]    \\
  {S_i}' /S'  \ar[r]  &   T' / S'
}
\end{displaymath}
  The two horizontal arrows are admissible monics (i.e., closed embeddings  of LCA groups), and the right vertical arrow is a morphism in LCA by our condition.
  It follows that the left vertical arrow is also a morphism in LCA.
 Now consider the commutative diagram for $i=1 , 2$,
 \begin{displaymath}
\xymatrix{
 0 \ar[r] &   S_i/ S  \ar[r] \ar[d] &    T / S   \ar[r]  \ar[d]  &  T / S_i    \ar [r]  \ar[d] & 0    \\
   0 \ar[r] &   {S_i}'/ S'  \ar[r]  &    T' / S'   \ar[r]    &  T' / {S_i}'    \ar [r]   & 0
 }
\end{displaymath}
 The two horizontal rows are exact in LCA,  and the first two vertical arrow  are morphisms in LCA. It follows that the third vertical arrow is a morphism in LCA.
 Then we can use the diagram
  \begin{displaymath}
\xymatrix{
 0 \ar[r] &  S_2/ S_1  \ar[r] \ar[d] &    T / S_1   \ar[r]  \ar[d]  &  T / S_2    \ar [r]  \ar[d] & 0    \\
   0 \ar[r] &   {S_2}'/ {S_1}'  \ar[r]  &    T' / {S_1}'   \ar[r]    &  T' / {S_2}'    \ar [r]   & 0
 }
\end{displaymath}
 to prove that $S_2/ S_1 \to  {S_2}'/ {S_1}' $ is a morphism in LCA.
\end{proof}

Lemma \ref{lemma3.3} implies that

\begin{lemma}\label{lemma3.4}   Let $G, G'$ be objects in LCA(2). Assume that ${\cal B}$ (resp. ${\cal B}'$) is a basis for $G$ (resp. $G'$).
  A group homomorphism $f : G \to G'$ is a morphism in LCA(2) iff it satisfies the following conditions:
\newline (1)  For every $  S' \in {\cal B}'  $, there is $S\in  {\cal B} $ such that $ f ( S ) \subset S'$.
\newline (2)  For every $  S \in {\cal B}  $, there is $S'\in  {\cal B}' $ such that $ f ( S ) \subset S'$.
\newline (3)  For $S_1 , S_2 \in  {\cal B}$ with $S_1 \subset S_2 $ and for $S_1' , S_2' \in  {\cal B}' $ with $S_1' \subset S_2' $,
 if $f ( S_1 ) \subset S_1' $ and $f(S_2 ) \subset S_2 '$, the induced map
   \[    S_2 / S_1 \to   S_2' / S_1' \]
 is a morphism in LCA.
\end{lemma}

\noindent {\bf Example 3.1.3.}  Let $k$ be a local field.  As in Example 3.1.2, $k((t))^n $ has an LCA(2) structure with a basis $\{  t^i k [[t]]^n \, | \, i \in {\Bbb Z}  \} $.
  Let $ f :  k((t))^m  \to k (( t ))^n $ be a $k(( t))$-linear map given by the $n \times m $ matrix $ (t_{ij})\in M_{n,m} ( k(( t))) $.
 There exists $N$ such that $  t^N  ( t_{ij} ) \in M_{n , m } ( F[[t]] ) $, hence   $f (   t^i k[[ t]]^m ) \subset   t^{i-N} k [[ t]]^n $.
  From this we see that the conditions (1) and (2) in Lemma \ref{lemma3.4} are satisfied.  The condition (3)
  is also satisfied since the map  $ \bar f :  S_2 / S_1 \to   S_2' / S_1' $ is a $k$-linear map of finite dimensional $k$-spaces, hence continuous.
 Therefore $f$ is a morphism in LCA(2) by Lemma \ref{lemma3.4}.      In particular  $ GL_n ( k(( t)))$ acts on $ k((t))^n$ as automorphisms in LCA(2).
        Thus we obtain a functor from    $k(( t))$-Mod,   the category of finite dimensional vector spaces over field $k((t))$, to LCA(2).

\begin{lemma}\label{lemma3.5}   The biproduct exists
for every pair of objects in LCA(2).
\end{lemma}

\begin{proof} (sketch)  Let $(G , {\cal U } ( G ) ) $ and $(H , {\cal U } ( H ) ) $ be objects in LCA(2).
   Let $ G\times H$ be the direct sum of $G$ and $H$. The family of subgroups
    \[ {\cal U} = \{   U \times V \; | \;  U \in {\cal U } ( G ) , V \in {\cal U } ( H ) \} \]
satisfies the conditions in Lemma \ref{lemma3.1}, so we have an LCA(2) structure on $A\times B $ with ${\cal U}$ as a basis.
 We have the obvious diagram
 \[   G  { {  \stackrel { p_1}  { \longleftarrow }} \atop  { {\longrightarrow} \atop { i_1} } }   G \times H  { {  \stackrel { p_2}  { \longrightarrow }}
 \atop  { {\longleftarrow} \atop { i_2} } }   H  \]
with arrows $p_1 , p_2 , i_1 , i_2 $ satisfying the identities
\[   p_1 i_1 = 1_A , \; \; \; p_2 i_2 = 1_B, \; \; \;  i_1 p_1 + i_2 p_2 = 1_C .\]
It is easy to see all the four maps are morphisms in LCA(2).
\end{proof}

One can easily verify that LCA(2) is an $Ab$-category.
Hence from the last lemma it follows that LCA(2) is an additive category.

\noindent {\bf Example 3.1.4.}  If $G \in {\rm LCA}$, then we can view $G$ as an LCA(2) group with ${\cal U} ( G ) $ as the set of
 all closed subgroup groups. The two-element set  $\{ \{e \} ,   G \}$ is a basis for $( G ,  {\cal U} ( G ) )$.
 For $G_1 , G_2 \in {\rm LCA}$,   a map $ f :   G_1 \to G_2 $ is a morphism in LCA iff it is a morphism in LCA(2).
 Therefore LCA is a full subcategory of LCA(2).

\subsection{\bf Exact category structure of LCA(2)} In this section we introduce a class of admissible monics and admissible epics in LCA(2), which gives rise to an exact category structure. The proof consists of a long series of lemmas, and we
strongly recommend the readers to skip over them to the main result Theorem \ref{exact} at  first reading.

A morphism $f : G_1 \to G_2$ in LCA(2) is called an {\bf admissible monic} if
\newline
M1.  $f$ is a monic of abelian groups (so $G_1$ is a subgroup of $G_2$).
\newline  M2. $  \bigcap_{S\in {\cal U } ( G_2) } ( G_1 + S ) = G_1 $.
\newline M3.  There is a basis $ {\cal B} $ of $G_2$ such that $ \{ S \cap G_1 \, | \,  S\in {\cal B} \}$ is a basis for $G_1$, and
  for any $S_1 \subset S_2$ in ${\cal B}$ the map $  S_2 \cap G_1 /  S_1 \cap G_1 \to S_2 / S_1 $ is an admissible
 monic in LCA, i.e., a closed embedding.

\begin{lemma}\label{lemma3.6}  Let $f : G_1 \to G_2$ be an admissible  monic in LCA(2). Then there exists a basis ${\cal B}$ of $G_2$ such that
M3 holds for ${\cal B}$ and $ \{ S \cap G_1 \, | \,  S\in {\cal B} \} =  {\cal U} ( G_1 )$.
\end{lemma}

This lemma will be proved using the Zorn's lemma and the following

\begin{lemma}\label{lemma3.7}  Let $f : G_1 \to G_2$ be a morphism in LCA(2), and ${\cal B}$ be a basis of $G_2$. Suppose that
 $(f , {\cal B})$ satisfies M1,  M2 and M3.  For $S_1 \subset S_2$ in ${\cal B}$
  and $ T \in {\cal U} ( G_1) $  satisfying $    S_1 \cap G_1 \subset T \subset S_2\cap G_1 $,  we can always choose
    $S\in {\cal U}(G_2)$ such that  $S_1 \subset S \subset S_2 $ and
 $ S \cap G_1 = T $. Moreover ${\cal B}' = {\cal B} \cup \{ S\} $ is also a basis of $G_2$, and  $(f , {\cal B}')$ satisfies  M1, M2 and M3.
  \end{lemma}

\begin{proof}
    Since $  S_2 \cap G_1 / S_1 \cap G_1 \to  S_2 / S_1 $ is a closed embedding, by the completeness axiom LCA(2)-4,  we can find $ S\in {\cal U} (G_2) $ with $S_1 \subset S \subset S_2$
    such that $ T / S_1 \cap G_1 \to  S/S_1 $ is an isomorphism.  This map factors through $S\cap G_1/S_1\cap G_1$, from which one deduces that  $ S\cap G_1 = T$.   Let ${\cal B}' = {\cal B} \cup \{ S \} $.
     Clearly ${\cal B}'$ is a basis and $(f, {\cal B}')$ satisfies M1 and  M2.
      We want to prove that ${\cal B}'$ satisfies  M3 as well, i.e., $U_2 \cap G_1  / U_1 \cap G_1 \to U_2/ U_1 $ is a closed embedding for any $U_1 \subset U_2$ in  ${\cal B}'$. It suffices to prove the case that one of $U_1$ and $U_2$ is $S$. Assume first that $U_1=S.$ Then we have three subcases.

 {\bf Case 1.}
   $U_1 = S$ and  $ U_2 = S_2$.  We have the commutative diagram
 \begin{displaymath}
\xymatrix{
  S \cap G_1 /  S_1 \cap G_1   \ar[r]  \ar[d]  &  S_2 \cap G_1 / S_1 \cap G_1   \ar[r]  \ar[d]  &  S_2 \cap G_1 /  S\cap G_1  \ar[d]
    \\
  S/ S_1   \ar[r]  &  S_2  / S_1   \ar[r]  &   S_2 / S
}
\end{displaymath}
   The first vertical arrow is an isomorphism,  the second vertical arrow is a closed embedding,  and the two rows are exact triples in LCA. It follows that the third
    vertical is also a closed embedding.

 {\bf Case 2.}    $U_1 = S$ and  $ U_2 \supset  S_2$. We have the commutative diagram
\begin{displaymath}
\xymatrix{
  S_2 \cap G_1   /  S \cap G_1   \ar[r]  \ar[d]  &  U_2 \cap G_1   / S \cap G_1    \ar[r] \ar[d] & U_2 \cap G_1 / S_2 \cap G_1  \ar[d]
    \\
    S_2 /  S   \ar[r]   &  U_2    / S   \ar[r]   & U_2  / S_2
}
\end{displaymath}
The rows are exact triples in LCA, and the first and third vertical arrows are admissible monics in LCA. By the standard result in exact category (see e.g.
\cite[page 10]{B1}), the middle vertical arrow is an admissible monic as well.

 {\bf Case 3.}   $U_1 = S$ and $S\subset U_2$.  This case is more general than Case 2.
 We choose  $S' \in {\cal B} $ such that $   S \subset U_2 \subset S'$,   $S_2 \subset S'$. We have a
  commutative diagram
\begin{displaymath}
\xymatrix{
  U_2 \cap G_1   /  S \cap G_1   \ar[r]  \ar[d]  &  S' \cap G_1   / S \cap G_1     \ar[d]
    \\
    U_2 /  S   \ar[r]   & S'    / S
}
\end{displaymath}
The two horizontal arrows are closed embeddings. Since $S_2 \subset S'$, by Case 2 the right vertical arrow is also a closed embedding. It follows that the left vertical arrow is
 a closed embedding.

 The case $U_2 = S$
can be handled similarly. \end{proof}

\begin{proof} (of Lemma \ref{lemma3.6})   Let $S$ be the set of bases ${\cal B}$ satisfying M3.
Then ${ S}$ is partially ordered by inclusion.  If $C$ is a totally ordered subset of $S$, then it is clear that $ \bigcup_{ {\cal B} \in {C}} {\cal B} $
 is also in $S$.  This proves that every totally ordered subset of $S$ has a upper bound in $S$. By Zorn's lemma, $S$ has a maximal element which must satisfy the properties in Lemma \ref{lemma3.6} (otherwise we get a contradiction using Lemma \ref{lemma3.7}). \end{proof}

\begin{lemma}\label{lemma3.8} Let $f : G_1 \to G_2$ be an admissible  monic in LCA(2). Then for any basis ${\cal B}_1$ of $G_1$, there exists a basis ${\cal B}_2$ of $G_2$ such that M3 holds for ${\cal B}_2$ and $ \{ S \cap G_1 \; | \;  S\in {\cal B}_2 \}={\cal B}_1$.
\end{lemma}

\begin{proof} Let ${\cal B}$ be the basis of $G_2$ as in Lemma \ref{lemma3.6}, and let
\[
{\cal B}_2=\{S\in{\cal B}\;|\; S\cap G_1\in{\cal B}_1\}.
\]
Then it suffices to show that ${\cal B}_2$ is a basis. Let $S\in {\cal B}$, and $T=S\cap G_1\in{\cal U}(G_1)$. Choose $T_1, T_2\in {\cal B}_1$ such that $T_1\subset T\subset T_2$. Then there exist $S_1, S_2\in{\cal B}_2$ such that $S_i\cap G_1=T_i$, $i=1, 2$.
Take $S_1'\in{\cal B}$ which is contained in both $S$ and $S_1$, and let $T_1'=S_1'\cap G_1\in{\cal U}(G_1)$. We have closed embeddings
\begin{displaymath}
\xymatrix{
T_1/T_1' \ar[r] & T/T_1' \ar[r] & S/S_1'.
}
\end{displaymath}
Viewing $T_1/T_1'$ as a closed subgroup of $S/S_1'$, we denote its preimage under the projection $S\to S/S_1'$ by $S_1''$. Then we have $S_1''\cap G_1=T_1$, hence $S_1''\in{\cal B}_2$ and $S_1''\subset S$. Similarly one may find $S_2''\in {\cal B}_2$ such that $S\subset S_2''$.
\end{proof}

\begin{lemma}\label{lemma3.9}  The composition of two admissible monics in LCA(2) is also an admissible monic.
\end{lemma}

\begin{proof} Let $f_1: G_1\to G_2$ and $f_2: G_2\to G_3$ be two admissible monics and let $f=f_2\circ f_1$. We need to show $f$ is an admissible monic. M1 is clear.  To prove M2, first notice that
\[
\bigcap_{S\in {\cal U}(G_3)}(G_1+S)\subset \bigcap_{S\in {\cal U}(G_3)}(G_2+S)=G_2.
\]
Then  we obtain
\[
\bigcap_{S\in {\cal U}(G_3)}(G_1+S)= \bigcap_{S\in {\cal U}(G_3)}( (G_1+S) \bigcap G_2 )   =  \bigcap_{S\in {\cal U}(G_3)}(G_1+S\cap G_2)=G_1.
\]
Let ${\cal B}_2$ be the basis of $G_2$ satisfying M3 such that $\{S\cap G_1\; |\; S\in{\cal B}_2\}={\cal U}(G_1)$. By Lemma \ref{lemma3.8}, we may find a basis ${\cal B}_3$ of $G_3$ satisfying M3 with respect to $f_2$ such that $\{S\cap G_2\;|\; S\in{\cal B}_3\}={\cal B}_2$. Then ${\cal B}_3$ also satisfies M3 with respect to $f$.
\end{proof}

Now we turn to admissible epics. An {\bf admissible epic} in LCA(2)  is a morphism $f : G_1 \to G_2$ such that
 \newline
E1. $f$ is  a surjective map.
 \newline
  E2. There is a basis ${\cal B}$ of $G_1$ such that
$\{  f ( S) \, | \,  S \in {\cal B}\}$ is a basis for $G_2$, and for any $S_1 \subset S_2 $ in ${\cal B}$,  the map
 $ S_2 / S_1  \to f(S_2 ) / f(S_1 ) $ is an admissible epic in LCA.

\begin{lemma}\label{lemma3.10}  Let $f : G_1 \to G_2$ be an admissible epic in LCA(2). Then there exists a basis ${\cal B}$ of $G_1$ such that
 {E2} holds for ${\cal B}$ and $ \{ f(S) \, | \,  S\in {\cal B} \} =  {\cal U} ( G_2 )$.
\end{lemma}

Similar to  Lemma \ref{lemma3.6}, this lemma follows from Zorn's lemma and the following:

\begin{lemma}\label{lemma3.11}  Let $f : G_1 \to G_2$ be a morphism in LCA(2) satisfying E1 and E2. For every $ T \in {\cal U} ( G_2) $ not in $\{ f(S) \, | \,  S \in  {\cal B} \}$, we can increase ${\cal B}$ to ${\cal B}' = {\cal B} \cup \{ S\} $ by adding one element $S\in {\cal U}(G_1)$ so that
 $ f(S) = T $ and E2 holds for ${\cal B}'$.
 \end{lemma}

\begin{proof} Choose $S_1 , S_2 \in {\cal B}$ such that $S_1 \subset S_2$ and $  f(S_1) \subset T \subset f(S_2)$.
    Since $ S_2 / S_1 \to  f(S_2) / f(S_1) $ is an admissible epic, by the completeness axiom LCA(2)-4 we have $S\in {\cal B}$, where $S_1\subset S\subset S_2$ such that $S/S_1$ is the preimage of $T/f(S_1)$, and $ S / S_1 \to  T/f(S_1) $ is an admissible epic.   It is easy to prove that $ f(S) = T$.   Let ${\cal B}' = {\cal B} \cup \{ S \} $. We want to prove that ${\cal B}'$ satisfies E2, i.e., for any $U_1 \subset U_2$ in ${\cal B}'$ the map
 $  U_2   / U_1  \to f(U_2)/ f(U_1) $ is an admissible epic. It suffices to prove the case that one of $U_1$ and $U_2$ is $S$. Assume first that $U_1=S.$

 {\bf Case 1.}
   $U_1 = S$ and  $ U_2 = S_2$.  We have the commutative diagram
 \begin{displaymath}
\xymatrix{
  S /  S_1    \ar[r]  \ar[d]  &  S_2  / S_1   \ar[r]  \ar[d]  &  S_2  /  S  \ar[d]
    \\
  T/ f(S_1)   \ar[r]  &  f(S_2)  / f(S_1)   \ar[r]  &   f(S_2) / T
}
\end{displaymath}
   The first two vertical arrows are admissible epics,  and the two rows are exact triples in LCA. It follows that the third
    vertical is also an admissible epic.

 {\bf Case 2.}    $U_1 = S$ and  $ U_2 \supset  S_2$. We have the commutative diagram
\begin{displaymath}
\xymatrix{
  S_2   /  S   \ar[r]  \ar[d]  &  U_2    / S    \ar[r] \ar[d] & U_2  / S_2   \ar[d]
    \\
    f(S_2) /  T   \ar[r]   &  f(U_2)    / T   \ar[r]   & f(U_2)  / f(S_2)
}
\end{displaymath}
The rows are exact triples in LCA, and the first and third vertical arrows are admissible epics in LCA. Then the middle vertical arrow is also an admissible epic in LCA.

 {\bf Case 3.}   $U_1 = S$ and $S\subset U_2$.
 Choose  $S' \in {\cal B} $ such that $   S \subset U_2 \subset S'$,   $S_2 \subset S'$. We have the
  commutative diagram
\begin{displaymath}
\xymatrix{
  U_2   /  S    \ar[r]  \ar[d]  &  S'    / S     \ar[d]
    \\
    f(U_2) /  T   \ar[r]   & f(S')    / f(S)
}
\end{displaymath}
The two horizontal arrows are admissible monics. Since $S_2 \subset S'$, by Case 2 the right vertical arrow is also an
admissible epic. It follows that the left vertical arrow, which is already a surjective map,  is
 an admissible epic in LCA as well.

 Now assume that $U_2 = S$.

{\bf Case 4.}    $U_1 \subset S_1 $, $U_2 = S_2$. We have the commutative diagram
\begin{displaymath}
\xymatrix{
  S_1   /  U_1   \ar[r]  \ar[d]  &  S    / U_1     \ar[r] \ar[d] &  S  / S_1   \ar[d]
    \\
    f(S_1) /  f(U_1)   \ar[r]   &  f(S)    / f(U_1)    \ar[r]   &   f(S)  / f(S_1)
}
\end{displaymath}
  The two rows are exact in LCA, and the first and third vertical arrows are admissible epics in LCA. Hence the middle vertical arrow is also an admissible epic in LCA.

{\bf Case 5.}    $U_1 \subset U_2 = S_2$. Choose $S' \in {\cal B} $ such that $ S' \subset U_1 $ and $  S' \subset S_1 $. We have the commutative diagram
\begin{displaymath}
\xymatrix{
  U_1    /  S'    \ar[r]  \ar[d]  &  S    / S'     \ar[r] \ar[d] &  S  / U_1   \ar[d]
    \\
    f(U_1) /  f(S')   \ar[r]   &  f(S )    / f(S')    \ar[r]   &   f(S )  / f(U_1)
}
\end{displaymath}
 The two rows are exact in LCA.  The first vertical arrow is an admissible epic in LCA, and the second vertical arrow is also an admissible epic by Case 4, hence
  so is the third.
   \end{proof}

Using the same trick as in the proof of Lemma \ref{lemma3.8} and Lemma \ref{lemma3.9}, one can show
the following two lemmas.

\begin{lemma}\label{lemma3.12}  Let $f : G_1 \to G_2$ be an admissible epic in LCA(2). Then for any basis ${\cal B}_2$ of $G_2$, there exists a basis ${\cal B}_1$ of $G_1$ such that
 E2 holds for ${\cal B}_1$ and $ \{ f(S) \, | \, S\in {\cal B}_1 \} =  {\cal B}_2$.
\end{lemma}

\begin{lemma}\label{lemma3.13} The composition of two admissible epics in LCA(2) is also an admissible epic.
\end{lemma}

\begin{lemma}\label{lemma3.14}
If $f: G_1\to G_2$ is an admissible monic in LCA(2), then the cokernel of $f$ exists and is an admissible epic.
\end{lemma}

\begin{proof} Let ${\cal B}$ be a basis of $G_2$ satisfying M3, and let $\pi: G_2\to G_2/G_1$ be the natural projection. We will use Lemma \ref{lemma3.1} to give a unique LCA(2) structure on $G_2/G_1$ such that ${\cal B}':=\{\pi(S) \; | \; S\in {\cal B}\}$ is a basis. Assume that $\pi(S_1)\subset\pi(S_2)$ where $S_1, S_2\in{\cal B}$. Let us give an LCA structure on $\pi(S_2)/\pi(S_1)$.    If $ S_1 \subset S_2$, then  $ (S_2 \cap G_1) / (S_1 \cap G_1) $ is a closed subgroup of $S_2/S_1$,  and the abelian group
$\pi(S_2)/\pi(S_1)$ is isomorphic to the quotient group  $ \left( S_2/S_1 \right) /  \left( (S_2 \cap G_1) / (S_1 \cap G_1)\right)$, the latter is an LCA group,  it gives an LCA group
 structure on $\pi(S_2)/\pi(S_1)$ vis isomorphism.
  It is clear that if $S_1 \subset S_2 \subset S_3$ with $S_i\in {\cal B}$, then $ \pi (S_2) / \pi (S_1 ) \to \pi (S_3) / \pi ( S_1 )$ is an monic in LCA.
    In general case,   we choose $S, S'\subset {\cal B}$ such that $S\subset S_i\subset S'$, $i=1, 2$.
  Then $  \pi ( S_1 ) / \pi ( S ) $ and $ \pi (S_2 ) / \pi ( S ) $ are both closed subgroups of $ \pi ( S' ) / \pi ( S )$, the condition that $ \pi ( S_1 ) \subset \pi ( S_2)$
    implies there is a closed embedding $  \pi ( S_2 ) / \pi ( S )  \to  \pi ( S_1 ) / \pi ( S )$.
   We give $ \pi ( S_2 ) / \pi (S_1 ) =   \left( \pi ( S_2 ) / \pi ( S ) \right)  /  \left(  \pi ( S_1 ) / \pi ( S )\right) $ the quotient topology. This topology is independent of the choices of $S$ and $S'$. It is easy to check that the conditions
     of Lemma \ref{lemma3.1} hold for $G_2/ G_1$ and ${\cal B}'$,  so we have an LCA(2)
     structure on $ G_2 / G_1$. By the construction,  the map $\pi : G_2 \to G_2 / G_1$
      is an admissible epic in LCA(2).     It is straightforward to prove that $\pi$ is the cokernel of $f$.
\end{proof}

\begin{lemma}\label{lemma3.15}
If $f: G_1\to G_2$ is an admissible epic in LCA(2), then the kernel of $f$ exists and is an admissible monic.
\end{lemma}

\begin{proof} Let ${\cal B}$ be a basis of $G_1$ satisfying E2. Let $G_0$ be the set-theoretical kernel of $f$, and let $i: G_0\to G_1$ be the inclusion. We will use Lemma \ref{lemma3.1} to give a unique LCA(2) structure on $G_0$ such that ${\cal B}':=\{S\cap G_0\; |\; S\in{\cal B}\}$ is a basis. Assume that $S_1\cap G_0\subset S_2\cap G_0$ where $S_1, S_2\in{\cal B}$. Choose $S, S'\in{\cal B}$ such that $S\subset S_i\subset S'$, $i=1, 2$. $S_i\cap G_0/S\cap G_0$ are closed subgroups of $S'/S$, $i=1, 2$, hence $S_1\cap G_0/S\cap G_0$ is a closed subgroup of $S_2\cap G_0/S\cap G_0$ and we give the quotient LCA structure on $S_2\cap G_0/S_1\cap G_0$, which is independent of the choice of $S$ and $S'$. It is easy to check that LCA(2)-1,  LCA(2)-2 and LCA(2)-3 hold for ${\cal B}'$, hence Lemma \ref{lemma3.1} applies. We have a unique LCA(2) structure on $G_0$ with ${\cal B}'$
 as a basis. And
$i$ is clearly an admissible monic. In fact M2 follows from the fact $\cap_{S\in{\cal B}}f(S)=\{0\}$, and M3 follows from E2.

To show that $i$ is the kernel of $f$, consider a morphism $g: G\to G_1$ in LCA(2) such that $f\circ g=0$. Let $h: G\to G_0$ be the unique map such that $g= i\circ h$. To show $h$ is a morphism we may apply Lemma \ref{lemma3.4} to bases ${\cal U}(G)$ and ${\cal B}'$. The conditions (1) and (2) in Lemma \ref{lemma3.4}
 are easy to check. Assume $T_1, T_2\in {\cal U}(G)$ with $T_1\subset T_2$, $S_1, S_2\in{\cal B}$ with $S_1\cap G_0\subset S_2\cap G_0$ such that $h(T_i)\subset S_i\cap G_0$, $i=1, 2$. Let $S'\in{\cal B}$ be as above. From the following commutative diagram
\begin{displaymath}
\xymatrix{
 S_2\cap G_0/S_1\cap G_0 \ar[r] & S'\cap G_0/ S_1\cap G_0 \ar[r] & S'\cap G_0/ S_2\cap G_0\\
 T_2/T_1 \ar@{.>}[u]^{\exists |} \ar[ru] \ar@/_/[rru]_0 &
}
\end{displaymath}
we see that $T_2/T_1\to S_2\cap G_0/S_1\cap G_0$ is a morphism in LCA. Note that the first row is an exact triple in LCA. This verifies (3) of Lemma \ref{lemma3.4}.
\end{proof}

\begin{lemma} \label{lemma3.16}
Let
\begin{displaymath}
\xymatrix{
 G_1\ar[r]^i & G_2\ar[r]^\pi & G_3
}
\end{displaymath}
be an admissible triple in LCA(2).
\newline (1) For any morphism $ f : B \to G_3$,  there exists  a diagram
\begin{displaymath}
\xymatrix{
  G_1  \ar[r]^{i'}  \ar[d]^=  &  G_2'  \ar[r]^{\pi'}  \ar[d]^{f'}
&  B  \ar[d]^f    \\
  G_1  \ar[r]^{i}  &  G_2  \ar[r]^{\pi}  &  G_3
}
\end{displaymath}
in LCA(2) such that the right square is a pullback and the first row is also an admissible triple.
\newline (2) For any morphism $ f : G_1 \to C$, there exists
 a diagram
\begin{displaymath}
\xymatrix{
  G_1  \ar[r]^{i}  \ar[d]^f  &  G_2  \ar[r]^{\pi}  \ar[d]^{f'}
&  G_3   \ar[d]^{=}    \\
  C  \ar[r]^{i'}  &  G_2'  \ar[r]^{\pi'}  &  G_3
}
\end{displaymath}
in LCA(2) such that the left square is a pushout and the second row is also an admissible triple.
\end{lemma}

\begin{proof} (1) Consider the morphism
 \[
 \pi_f: G_2\times B\to G_3,\quad (a,b)\mapsto \pi(a)-f(b).
  \]
  Then $\pi_f$ is onto. Let ${\cal B}$ be a basis of $G_2$ such that E2 holds for $\pi.$ Let \[{\cal B}'=\{S\times T\; |\; S\in {\cal B}, T\in {\cal U}(B), f(T)\subset \pi (S)\}.\]
  It is easy to see that ${\cal B}'$ is a basis of $G_2\times B$, and
  \[
  \{ \pi_f(S\times T)\; |\; S\times T\in{\cal B}'\}=\{\pi(S)\;|\; S\in{\cal B}\}
  \]
  is a basis of $G_3$. For $S_1\times T_1\subset S_2\times T_2$, which are both in ${\cal B}'$, the induced map
  \[
  \bar{\pi}_f: S_2\times S_1/ T_2\times T_1\to \pi(S_2)/\pi(S_1)
  \]
  is an admissible epic in LCA. Hence $\pi_f$ is an admissible epic in LCA(2) and we denote
  $Ker(\pi_f)$ by $i_f: G_2' \to G_2\times B$. Let $f' : G_2' \to G_2$ and
 $\pi' : G_2' \to B$ be the obvious projections.  $i' : G_1 \to G_2 ' $ is given by
$a_1 \mapsto (a_1 , 0 )$. The diagram in (1) is commutative.  Let us prove that the right square is a pullback. Assume we have the following commutative diagram in LCA(2)
\begin{displaymath}
\xymatrix{
G \ar@{.>}[rd]^h \ar@/^/[rrd]^{g_1} \ar@/_/[ddr]^{g_2} &  & \\
& G_2'  \ar[r]^{\pi'}  \ar[d]^{f'}
&  B  \ar[d]^f \\
& G_2  \ar[r]^{\pi}  &  G_3
}
\end{displaymath}
Then the unique map $h: G \to G_2'$ making the whole diagram commutative must be a morphism. This is because $G_2'=Ker(\pi_f)$ and
the composition
\begin{displaymath}
\xymatrix{
G \ar[r]^{g_2\times g_1} & G_2\times B \ar[r]^{\pi_f} & G_3
}
\end{displaymath}
is zero. It remains to show the first row in (1) is an admissible triple. From the proof of Lemma \ref{lemma3.15} we know that ${\cal B}'_2:=\{S\times_{G_3}T \; |\; S\times T\in{\cal B}'\}$
is a basis of $G_2'.$ It is easy to check that
$\{\pi'(S\times_{G_3}T)\;|\; S\times T\in{\cal B}'\}={\cal U}(B)$.
Assume $S_1\times_{G_3}T_1\subset S_2\times_{G_3}T_2$ which are both
in ${\cal B}_2'$. Then $T_1\subset T_2.$ Choose $S'\in{\cal B}$ containing both $S_1$ and $S_2.$ Then
$S'\times T_2\in {\cal B}'$, and we have the following square of pullback in LCA
\begin{displaymath}
\xymatrix{
S'\times_{G_3}T_2/ S_1\times_{G_3}T_1 \ar[d]^{\bar{f}'} \ar[r]^{\qquad\quad\bar{\pi}'} & T_2/T_1\ar[d]^{\bar{f}}\\
S'/S_1\ar[r]^{\bar{\pi}} & \pi(S')/\pi(S_1)
}
\end{displaymath}
The bottom arrow $\bar{\pi}$ is an admissible epic, hence so is the top arrow $\bar{\pi}'$.
Since $S_2\times_{G_3}T_2/ S_1\times_{G_3}T_1$ is a closed subgroup of $S'\times_{G_3}T_2/ S_1\times_{G_3}T_1$, the induced map
\[
S_2\times_{G_3}T_2/S_1\times_{G_3}T_1\to T_2/T_1,
\]
being onto, must be also an admissible epic. This proves that $\pi'$ is an admissible epic in LCA(2).
It is not hard to prove that $i'$ is the kernel of $\pi'$.
\newline (2) Consider the morphism
\[
i_f: G_1\to G_2\times C, \quad a\mapsto (i(a), - f(a)).
\]
Then $i_f$ satisfies M1 and M2. Let ${\cal B}$ be a basis of $G_2$ such that M3 holds for $i$. Let
\[
{\cal B}'=\{S\times T\; |\; S\in{\cal B}, T\in{\cal U}(C), f(S\cap G_1)\subset T\}.
\]
Then ${\cal B}'$ is a basis of $G_2\times C$ and
\[
\{i_f^{-1}(S\times T)\; |\; S\times T\in{\cal B}'\}=\{S\cap G_1\; |\; S\in{\cal B}\}
\]
is a basis of $G_1$. For $S_1\times T_1\subset S_2\times T_2$ which are both in ${\cal B}'$, the induced map
\[
\bar{i}_f: S_2\cap G_1/ S_1\cap G_1\to S_2/S_1\times T_2/T_1
\]
is a closed embedding. Hence $i_f$ is an admissible monic in LCA(2), and we denote $Coker(i_f)$  by
$\pi_f: G_2\times C\to G_2'$. Let $f': G_2\to G_2'$, $i: C\to G_2'$ be the natural maps. $\pi'$ is given by
$\pi_f(a, c)\mapsto \pi(a)$, where $(a,c)\in G_2\times C$. The left square in the diagram of (2) is a pushout. Indeed assume that we have the commutative diagram in LCA(2)
\begin{displaymath}
\xymatrix{
G_1 \ar[r]^i \ar[d]^f & G_2 \ar[d]^{f'} \ar@/^/[rdd]^{g_2} & \\
C \ar[r]^{i'}  \ar@/_/[rrd]^{g_1}  &  G_2' \ar@{.>}[rd]^h & \\
& & G
}
\end{displaymath}
The unique map $h$ making the whole diagram commutative is a morphism, since
$G_2'=Coker(i_f)$ and the composition
\begin{displaymath}
\xymatrix{
G_1 \ar[r]^{i_f} & G_2\times C \ar[r]^{g_2\oplus g_1} & G
}
\end{displaymath}
is zero. Let us prove that the second row in (2) is an admissible triple. From the proof of Lemma
\ref{lemma3.14} we have ${\cal B}_2':=\{\pi_f(S\times T)\; |\; S\times T\in {\cal B}'\}$ is a basis
of $G_2'$. We have $\bigcap_{U\in {\cal B}_2'}(i'(C)+U)=i'(C)$, which can be checked by using {M2} for $i$, i.e., using $\bigcap_{S\in {\cal B}}(G_1+S)=G_1$. This proves M2 for $i'$.
It is easy to see $\{i'^{-1}\pi_f(S\times T)\; |\; S\times T\in{\cal B}'\}={\cal U}(C)$.
Assume $\pi_f(S_1\times T_1)\subset \pi_f(S_2\times T_2)$ which are both in ${\cal B}_2'$. Then $T_1\subset T_2$. Choose $S\in{\cal B}$ contained in both $S_1$ and $S_2$. Then $S\times T_1\in{\cal B}'$, and we have the following square of pushout in LCA
\begin{displaymath}
\xymatrix{
S_2\cap G_1/S\cap G_1 \ar[r]^{\bar{i}} \ar[d]^{\bar{f}} & S_2/S \ar[d]^{\bar{f}'}\\
T_2/T_1 \ar[r]^{\bar{i}'\quad\qquad} & \pi_f(S_2\times T_2)/\pi_f(S\times T_1)
}
\end{displaymath}
The top arrow $\bar{i}$ is a closed embedding, hence so is the bottom arrow $\bar{i}'$. Since
$\pi_f(S_2\times T_2)/\pi_f(S_1\times T_1)$ is a quotient of $\pi_f(S_2\times T_2)/\pi_f(S\times T_1)$,
the induced map
\[
T_2/T_1\to \pi_f(S_2\times T_2)/\pi_f(S_1\times T_1),
\]
being injective, must be also a closed embedding. This proves that $i'$ is an admissible monic in LCA(2). One can verify easily that $\pi'$ is the cokernel of $i'$.
\end{proof}

In summary we have proved the following result.

\begin{theorem}\label{exact}
  Let $\Sigma $ be the class of kernel-cokernel pairs $(i , p )$ in LCA(2)
       \begin{displaymath}
\xymatrix{
  G_1  \ar[r]^{i}    &  G_2  \ar[r]^{p}
&  G_3  }
\end{displaymath}
 such that $i$ is an admissible monic and $p$ is an admissible epic. Then
 $\Sigma $ is an exact category structure on LCA(2).
 \end{theorem}

\subsection{\bf  Bruhat-Schwartz functions on LCA(2) }\label{ss4.3}

For an LCA group $A$, we denote by $ \mu ( A) $ the space of invariant tempered distributions on $A$. Each nonzero Haar measure on $A$ is a basis vector of $\mu ( A ) $.
 We write the pairing of $\mu \in \mu ( A )$ with $ f \in {\cal S} ( A)$   as
 \[
 (\mu, f):=\int_A f( x)  d\mu ( x ) .
 \]
   If $i:  A_1 \to A_2$ is a closed embedding in LCA,  then the restriction map $i^* : {\cal S} ( A_2 ) \to {\cal S} ( A_1 )$
    is a homomorphism of locally convex topological spaces.
  If $ p: A \to A'$ is an admissible epic in LCA,  we have a map
 $p_* :  {\cal S} ( A ) \otimes \mu ( Ker ( p ) ) \to  {\cal S} ( A' ) $ given by the integration on the fiber of $p$:
  for $a' \in A' $, let $a \in p^{-1} ( a' ) $ and define
      \begin{equation}\label{3.3.1}  p_{*} ( f \otimes \mu ) ( a' ) =    \int_{ Ker ( p ) }  f ( a  +  x  )   d\mu ( x )  . \end{equation}
 There is an isomorphism
\begin{equation}\label{3.3.2}    \mu ( Ker ( p )  ) \otimes \mu ( A'  ) \to \mu ( A )  \end{equation}
given by
  \[  (   \mu  \otimes \mu'  , f ) = \int_{A'}  p_* ( f \otimes \mu  ) ( a' ) d\mu' ( a' ) . \]

\begin{lemma}\label{lemma3.3.1}
   (1) Let $i_1 : A_1 \to A_2$ and $i_2 : A_2 \to A_3$ be the closed embeddings in LCA. Then
  $ (i_2 i_1)^* =  i_1^* i_2^*$ as maps  from ${\cal S} ( A_3 )$ to ${\cal S} ( A_1) $.
 \newline (2)  Let $ p: A \to A'$ and $p':  A' \to A''$  be admissible epics in LCA.
       Then the kernel of $ p' p$ is $ Ker ( p' p ) = p^{-1} ( Ker ( p' ) ) $, and we have
 an exact triple
             \begin{displaymath}
\xymatrix{
  Ker ( p )   \ar[r]    &  Ker ( p' p ) = p^{-1} ( Ker ( p' ) )   \ar[r]
&  Ker ( p' ).   }
\end{displaymath}
Ientifying $ \mu (  Ker ( p' p ) ) $ with $\mu ( Ker ( p ) ) \otimes \mu ( Ker ( p ' ) ) $ by  (\ref{3.3.2}), one has
 $ (p' p)_* =  p' _* p_* $  as maps from $ {\cal S} ( A ) \otimes \mu ( Ker ( p'p) )$ to  ${\cal S} ( A'' )$ .
\end{lemma}

       Let $( G , {\cal U} )$ be an LCA(2) group.   For $S_1 ,S_2, T \in {\cal U}$ with $ T \subset S_1 , S_2 $,  the quotients $ S_1 / T $ and $S_2 / T$ are objects in LCA.
       Consider the one-dimensional space  $ \mu (S_1 / T )^\vee \otimes_{\Bbb C}  \mu ( S_2 / T) $. Suppose $T_1 \subset T$.
    Then the exact triple
    \[      T /   T_1  \to  S_i / T_1 \to  S_i / T \]
    gives an isomorphism $  \mu ( T / T_1 ) \otimes_{\Bbb C}   \mu ( S_i / T ) \to   \mu ( S_i  / T_1 ) $, hence isomorphisms
    \[  \mu ( S_1 / T )^\vee \otimes  \mu ( S_2 / T )  \to   \mu ( S_1 / T )^\vee \otimes   \mu ( T / T_1 )^\vee \otimes      \mu ( T / T_1 )\otimes  \mu ( S_2 / T )
       \to  \mu ( S_1 / T_1 )^\vee \otimes  \mu ( S_2 / T_1 ) . \]
     We define the one-dimensional space $ \mu ( S_1 , S_2 ) $ as
  \[  \mu ( S_1 , S_2 ) =   \lim_{\longrightarrow}   \mu (S_1 / T )^\vee \otimes_{\Bbb C}  \mu ( S_2 / T)     \]
   where the direct limit is taken over the filtered direct system $\{ T  \in {\cal U} \, | \,   T\subset  S_1 , S_2  \}$.
      By construction  the natural map
   $  \mu (S_1 / T )^\vee \otimes_{\Bbb C}  \mu ( S_2 / T) \to \mu ( S_1 , S_2 )$ is an isomorphism.
 For $S_1 , S_2 , S_3 \in {\cal U}$, there is a  natural isomorphism
     \begin{equation}\label{3.3.3}    \mu ( S_1 , S_2 ) \otimes  \mu ( S_2 , S_3 ) \to   \mu ( S_1 , S_3 ) .\end{equation}

   Fix $V_0 \in {\cal U} $.  We want to  define the Bruhat-Schwartz space relative to $V_0$ on $G$.  First for $V , U \in {\cal U} $ with $V \subset U$,
    we let
    \[  {\cal S}_{V_0} ( V , U )  := {\cal S} ( U  /  V ) \otimes \mu ( V , V_0 ), \]
    where ${\cal S} ( U  /  V )$ is the usual Bruhat-Schwartz space on the LCA group $U/ V$.
   For any $V_1 , V_2 , V_3 \in {\cal U}$ with  $ V_1 \subset V_2 \subset V_3$, we have an exact triple in LCA
    \begin{displaymath}
\xymatrix{
  V_2 / V_1  \ar[r]^{i}    &  V_3 / V_1  \ar[r]^{p}
&  V_3 / V_2 . }
\end{displaymath}
 So we have the map $ i^* :  {\cal S} (   V_3 / V_1 ) \to  {\cal S} (   V_2 / V_1 ) $, which induces
 \[ i^*  :  {\cal S}_{V_0} ( V_1 , V_3 ) \to  {\cal S}_{V_0} ( V_1 , V_2 ),    \; \; \;  f\otimes \mu  \mapsto  i^* ( f) \otimes \mu, \]
 and the map
 \[ p_*  :     {\cal S}_{V_0} ( V_1 , V_3 ) \to  {\cal S}_{V_0} ( V_2 , V_3 ), \; \; \;  f\otimes \mu  \mapsto  p_* ( f\otimes \mu_1) \otimes \mu_2 \]
 where we decompose $\mu \in \mu (V_1 , V_0 ) $ as $ \mu_1\otimes \mu_2 $ for $ \mu_1 \in \mu ( V_1 , V_2 ) $ and $\mu_2 \in  \mu ( V_2 , V_0 ) $
 using the isomorphism (\ref{3.3.3}).

\begin{lemma}\label{lemma3.3.2}
 If  $V_1 , V_2 , V_3 , V_4 \in {\cal U} $ satisfy $V_1 \subset V_2 \subset V_3 \subset V_4 $, then
 the following diagrams are commutative
  \begin{displaymath}
\xymatrix{
 {\cal S}_{V_0} ( V_1 , V_4 ) \ar[r]^{i^*} \ar[dr]^{i^*} & {\cal S}_{V_0} ( V_1 , V_3 ) \ar[d]^{i^*} \\
         & {\cal S}_{V_0} ( V_1 , V_2 )
}
\end{displaymath}
  \begin{displaymath}
\xymatrix{
 {\cal S}_{V_0} ( V_1 , V_4 ) \ar[r]^{p_*} \ar[dr]^{p_*} & {\cal S}_{V_0} ( V_2 , V_4 ) \ar[d]^{p_*} \\
         & {\cal S}_{V_0} ( V_3 , V_4 )
}
\end{displaymath}
 \begin{displaymath}
\xymatrix{
 {\cal S}_{V_0} ( V_1 , V_4 ) \ar[r]^{i^*} \ar[d]^{p_*} & {\cal S}_{V_0} ( V_1 , V_3 ) \ar[d]^{p_*} \\
    {\cal S}_{V_0} ( V_2 , V_4 )   \ar[r]^{i^*}    & {\cal S}_{V_0} ( V_2 , V_3 )
}
\end{displaymath}
  \end{lemma}

 This lemma follows directly from Lemma \ref{lemma3.3.1}.

 A family  $\{ h_{V , U } \}$ parametrized by  $V \subset U$ in  ${\cal U}$,
   where $  h_{V , U }   \in {\cal S}_{V_0} ( V, U)$,  is called a {\bf compatible family} if
    for any
    $V_1 \subset V_2 \subset V_3$ in ${\cal U}$,
   \begin{equation}\label{3c}
  i^* h_{V_1, V_3}  =h_{ V_1 , V_2 },  \; \; \; \;  p_*  h_{V_1 , V_3 }  =  h_{V_2 , V_3}.
   \end{equation}
  We denote by $    {\cal S}_{V_0} (G) $
   the space of all compatible families  $\{ h_{V , U } \}$ and call it the
  Bruhat-Schwartz space of $G$.
   It is clear that ${\cal S}_{V_0} (G)$ is a vector space.
   When the choice of $V_0$ is clear from the context, we write ${\cal S} ( G)$ for   ${\cal S}_{V_0} (G)$.
We have the projective maps $ {\cal S}_{V_0} (G)  \to   {\cal S}_{V_0} ( V , U ) :    \{ h_{V , U } \} \mapsto  h_{V , U } $.

  A subset  $ {\cal N} \subset {\cal U} \times {\cal U}$ is called a {\bf net}
   if  $(V_1 , V_2 ) \in {\cal N}$ implies that  $ V_1 \subset V_2$ and
  for any finitely many $S_1 , \dots , S_n $, there is $ (V_1 , V_2 )\in {\cal N}$ such that $ V_1 \subset S_1 , \dots , S_n \subset V_2 $.  The conditions
(\ref{3c}) imply that an element $\{ h_{V_1 , V_2 } \} \in {\cal S}_{V_0} (G)$ is determined by
 the subfamily  $\{ h_{V_1 , V_2 } \}$ over $(V_1 , V_2 ) \in {\cal N}$.
 Conversely  if we have a family of elements $h_{V_1 , V_2 } \in {\cal S} ( V_1 , V_2 )$ parametrized by $(V_1 , V_2 ) \in {\cal N}$ such that
   the compatiblility conditions (\ref{3c}) are satisfied, then the family can be extended to a unique element in  ${\cal S}_{V_0} (G)$ as follows.
    For any $V, U \in {\cal U} ( G ) $ with $ V \subset U$, choose $( V_1 , V_2 ) \in {\cal N} $ such that  $    V_1 \subset V \subset U \subset V_2$ and
     put
     \begin{equation}\label{3ext}
               h_{V , U } =    p_*  i^*     h_{V_1 , V_2 }
   \end{equation}
   with $i^* :   {\cal S}_{V_0} ( V_1, V_2  ) \to    {\cal S}_{V_0} ( V_1, U )$
     and $p_* :   {\cal S}_{V_0} ( V_1, U )  \to {\cal S}_{V_0} ( V, U)$.
      It follows from Lemma \ref{lemma3.3.2} that  $h_{V , U }$ is independent of the choice of $(V_1 , V_2 ) \in {\cal N}$ and
       $\{ h_{V, U} \} \in S_{V_0} ( G ) $.

 Next we introduce the dual object $G^*$ of an LCA(2) group $G$.  View $\Bbb T$ as LCA(2) group as in Example 3.1.4.
  It is easy to see that a group homomorphism $\chi :   G \to \Bbb T$ is a morphism in LCA(2) iff there is
  $ S \in {\cal U} ( G ) $ such that $ \chi |_{S} =1 $ and for every $ S' \in {\cal U} ( G ) $ with $ S \subset S' $,
   $\chi $ induces a continuous map $ S' / S \to \Bbb T$.
     Let $G^* = Mor ( G , \Bbb T ) $ with the obvious group structure. For each $S \in { \cal U } ( G ) $,
   let $ S^{\perp} = \{ \chi \in G^* \, | \; \chi|_S= 1 \}$, and $ {\cal U } ( G^* )=\{S^\perp \,|\, S\in \cal U(G)\}$.   If $S_1 \subset S_2 $, then $S_2^{\perp } \subset S_1^{\perp}$ and $  S_1^{\perp}  /  S_2^{\perp }  \cong ( S_2 / S_1 )^* $
   has an LCA structure. It is easy to check that $ ( G^* , {\cal U} ( G^* ) )$ is an LCA(2) group, which we call the dual group of $G$.

  Fix $V_0  \in {\cal U} ( G ) $.   For any $V\subset U$ in ${\cal U} ( G )$, we can define the Fourier transform
   \[   F:    {\cal S}_{V_0} ( V , U )  \to   {\cal S}_{V_0^{\perp} } ( U^{\perp} , V^{\perp} )
   \]
   as follows.
   For $ f \otimes \mu_1 \in  {\cal S}_{V_0} ( V , U ) = {\cal S} ( U / V ) \otimes \mu ( V, V_0 ) $,  we write
 $ \mu_1 = \mu \otimes \mu_2 $ where $\mu \in \mu ( V , U ) $, $\mu_2 \in \mu ( U , V_0 )$. Then
 \begin{equation}\label{3.3.10}
    F:   f \otimes \mu_1  \mapsto \int_{U/ V}  f ( x ) ( x, x^* )  d\mu ( x ) \otimes \mu_2 .
\end{equation}
 For any LCA group $A$ with dual group $A^*$, we have
 a pairing $ \mu ( A) \times \mu ( A^* ) \to {\Bbb C}$ defined by the equation
 \[    \int_{A^*} \left(  \int_A f ( x ) ( x , x^* ) d \mu  \right)d\mu^* = ( \mu , \mu^* ) f ( 0 ) \]
 for $f\in {\cal S} ( A ) $.  In this way we identify $\mu ( A) $ with $\mu ( A^* )^\vee$.
  This allows us to identify $ \mu ( V , U ) $ with $ \mu ( U^{\perp} , V^{\perp } ) $ and the identification is compatible with the
 isomorphisms $  \mu ( V , W ) \cong    \mu ( V , U ) \otimes \mu ( U , W )$,  $  \mu ( W^{\perp} , V^{\perp} ) \cong  \mu ( W^{\perp} , U^{\perp})
   \otimes \mu (  U^{\perp} , V^{\perp } ) $.
 The right hand side of (\ref{3.3.10}) is in $ {\cal S}_{V_0^{\perp} } ( U^{\perp} , V^{\perp} )$.
 If $ \{ h_{ V , U } \}  \in  {\cal S}_{V_0 } ( G ) $, then the family $  F h_{ V , U } \in   {\cal S}_{V_0^{\perp} } ( U^{\perp} , V^{\perp} )$
  also satisfies (\ref{3c}) for $G^*$, hence gives an element in ${\cal S}_{V_0^{\perp}} ( G ^* )  $. Therefore we have
 an isomorphism
\[  F:   {\cal S}_{V_0 } ( G ) \to {\cal S}_{V_0^{\perp}} ( G ^* ) , \]
which we call the Fourier transform.

Let ${\rm Aut } ( G )$ denote the automorphism group of $G$ as an LCA(2) group.
 Set
  \begin{equation}\label{3cext}
   \widetilde {\rm Aut } ( G ) = \{  ( \alpha , \lambda ) \; | \;  \alpha \in  {\rm Aut } ( G ) , \lambda \in \mu (  V_0 \alpha^{-1} , V_0 ) ,  \lambda \ne 0  \}
  \end{equation}
 with the composition given by
\begin{equation}\label{3com}
  ( \alpha_1 , \lambda_1 ) ( \alpha_2 , \lambda_2 ) = ( \alpha_1 \alpha_2 ,   \lambda_1 ( \lambda_2 \alpha_1^{-1} )  )
 \end{equation}
  where $ \lambda_2 \mapsto  \lambda_2 \cdot \alpha_1^{-1} $ is the obvious map  $ \mu ( V_0 \alpha_2^{-1} ,  V_0 ) \to
   \mu ( V_0 \alpha_2^{-1} \alpha_1^{-1} ,  V_0 \alpha_1^{-1}  )$
  induced by $\alpha_1^{-1}$.
  The group  $\widetilde {\rm Aut } ( G )$ is a central extension of $ {\rm Aut } ( G ) $ by ${\Bbb C}^\times $, and
it acts on ${\cal S} ( G ) $ as follows. Let $ \phi \in {\cal S} ( G ) $, whose $(V_1 , V_2 )$-component
   is  $\phi_{V_1 , V_2 } ( x ) \otimes \mu_{V_1, V_2 } \in {\cal S} ( V_2 / V_1 ) \otimes \mu ( V_1 , V_0 ) $. Let
$ \alpha' = ( \alpha , \lambda ) \in \widetilde {\rm Aut } ( G )$.   Then $ \pi ( \alpha' ) \phi $ has its $(V_1 , V_2 )$-component given by
  \begin{equation}\label{3act}
      \phi_{V_1\alpha , V_2\alpha  } ( x\alpha  ) \otimes   ( \mu_{V_1\alpha , V_2 \alpha  } \alpha^{-1} )  \lambda ,
      \end{equation}
 where $ \mu_{V_1\alpha, V_2 \alpha } \mapsto   \mu_{V_1\alpha, V_2 \alpha } \alpha^{-1} $ is
 the obvious map $ \mu ( V_1\alpha  , V_0 ) \to \mu ( V_1 , V_0 \alpha^{-1} ) $
 induced by $\alpha $.

  Suppose $f:  G_1 \to G_2 $ is an admissible monic in LCA(2). Then  there is a basis ${\cal B} \subset {\cal U} (G_2)$ such that
   $ \{  f^{-1} ( B ) \, | \, B \in {\cal B} \}$ is a basis of $G_1$, and for any $ U_1 \subset U_2$ in $ {\cal U} (G_2)$
   the map  $f:   f^{-1} ( U_2 )  / f^{-1} (  U_1 )  \to   U_2 / U_1$  is a closed embedding.  Let $V_0 \in {\cal B} $.
     Assume further that the image of $f^{-1}(U_2)/f^{-1}(U_1)$ is an open compact subgroup of
   $U_2 / U_1$.  In this case we can define the characteristic element $1_{G_1} $ of $G_1$ in ${\cal S}_{V_0 } ( G_2)$ as follows.
   It is clear that ${\cal N} = \{  ( V , U ) \in {\cal B}^2 \, | \, V  \subset V_0 \subset  U \} $ is a net of $G_2$.
    For each $ V \in {\cal B}$ with $ V \subset  V_0 $, let $ \mu_{G_1 , V }  \in  \mu ( V , V_0 )$ be the unique Haar measure
     of $ V_0 / V $ such that $\mu (   G_1 \cap V_0 /  G_1 \cap V )  = 1 $,   and $ 1_{ G_1 ,  U / V } \in {\cal S} ( U / V  ) $
       be the characteristic function of the open compact subgroup $  G_1 \cap U /  G_1 \cap V $ of  $U / V$. Then
        the family $ \{ 1_{ G_1 ,  U / V } \otimes   \mu_{G_1 , V } \in {\cal S}_{V_0} ( V , U ) \}$, as $(V, U ) $ runs through ${\cal N}$, is  compatible.
     We call the Bruhat-Schwartz function determined by this family the characteristic element of $G_1$ and denote it by  $1_{G_1} $.
One can show that it is independent of  $\cal B$.

   \section{ Weil index for quadratic character of  LCA(2)} \label{s5}

 In this section we generalize the results about LCA groups in Section \ref{s3}  to LCA(2) groups.

\subsection{Bicharacter and quadratic characters for  LCA(2)}

  Let $G, H $ be  objects in LCA(2). A map $f: G \times H\to \Bbb T$ is called a {\bf bicharacter} if (1) for each $y \in H$, $x\mapsto f(x , y )$ is a
 character of $G$, i.e., an LCA(2) morphism $G\to\Bbb T$,  and for each $x\in G$, $y \mapsto f( x , y ) $ is a character of $H$;
  (2)  for each $ V \in {\cal U} ( H ) $, there is $ U \in {\cal U} ( G)$
  such that $ f ( U \times V ) = 1$, and for  each $ U \in {\cal U} ( G ) $, there is $ V \in {\cal U} ( H)$ such that $ f ( U \times V ) = 1$.

  If  $ \alpha :  G \to H , x \mapsto x \alpha $ is a morphism in LCA(2), then we have the dual morphism
$\alpha^* : H^* \to G^*$ such that
\[   ( x \alpha , y^* ) =  ( x  , y^* \alpha^* ) \]
for every $x\in G, y^* \in H^* $.

  Let  $f: G\times H  \to \Bbb T$ be a bicharacter. For each $y \in H$,  $ x \mapsto f( x , y )$ is a character of $ G$, so we have a map
  $ \alpha : H \to G^* $ and it is easy to prove that $\alpha $ is a morphism in LCA(2). Similarly we have
  an LCA(2) morphism $ \alpha^* : G \to H^*$.  Then
  \[ f ( x , y ) = ( x , y \alpha )  = ( x \alpha^* , y ). \]
               If $G = H$, then $f$ is symmetric
  iff $\alpha = \alpha^*$.

    A map $ h : G \to \Bbb T$ is called a {\bf  quadratic character} if for every $U\in {\cal U}$, there is  $ V \in {\cal U} $ such that
    \begin{equation}\label{4.1}    h ( u + v ) = h ( u ) , \; \; \; {\rm for \; all } \; u \in U, \; v \in V \end{equation}
  and
  \[    ( x , y ) \mapsto h( x + y ) h(x)^{-1} h( y )^{-1}  \]
is a bicharacter on $G\times G$.  Then there is
 an LCA(2) morphism  $\rho : G \to G^* $ satisfying
   \begin{equation}\label{4.2}
  h( x + y ) h(x)^{-1} h( y )^{-1} = ( x , y \rho ) .
  \end{equation}
 We say $h$ is non-degenerate if $\rho  $ is an isomorphism from $G$ to $G^*$.

\begin{lemma}\label{lemma4.1}  Let $h$ be a non-degenerate quadratic character of $G$, and $\rho : G \to G^*$ be the isomorphism defined by (\ref{4.2}). Then the set of
 pairs $(U^\perp\rho^{-1}, U)$, $U\in{\cal U}$ satisfying

(1) $U^\perp\rho^{-1}\subset U$,

(2) $h ( u + v ) = h ( u ) $  for all $u\in U$ and $v\in U^\perp\rho^{-1}$
\\
form a net of ${\cal U}(G)$, and for any such pair  $h$ descends to a
non-degenerate quadratic character  on the LCA group $U/U^\perp\rho^{-1}$.
\end{lemma}

\begin{proof}   For any $S_1 , S_2\in\cal U(G) $ we need to find a pair $(U^\perp\rho^{-1}, U)$ as in the lemma
  such that $ U^\perp\rho^{-1} \subset    S_i  \subset  U$, $i=1,2$.  By Axiom LCA(2)-2 (Section \ref{ss4.1}),
 we may assume that $S_1 \subset S_2$.    By (\ref{4.1}),  there is $V\in {\cal U}$ such that $ V \subset S_2$
    and
  $h ( x + v ) = h ( x ) $ for all $ x \in S_2 $ and $v\in V$.    By Axiom LCA(2)-2, we may assume  that $ V \subset S_1$.
 The relation    $h ( x + v ) = h ( x ) $ implies that
   \[   h ( v ) ( x\rho  , v ) = 1 \; \; \; \textrm{ for all } \;   x \in S_2 , \ v \in V.  \]
 Taking $x=0$, we have $ h ( v ) = 1 $ for all $v\in V$, which implies that $( x\rho , v  ) = 1$.
 Hence $ S_2 \rho \subset V^{\perp} $,  $S_2 \subset V^{\perp} \rho^{-1} $.   Set $U = V^\perp\rho^{-1}$, then $ U^\perp\rho^{-1} = V $.
The pair $(U^\perp\rho^{-1}, U)$ satisfies all the required properties.
\end{proof}

 \subsection{Heisenberg group, automorphisms and actions on Bruhat-Schwartz space.}
 Let $G$ be an LCA(2) group.
   Similar to (\ref{2.2.1}), we define  $F:  (G \times G^* ) \times (G \times G^* ) \to \Bbb T$,
     \[  F( ( x_1 , x_1^* ) , (x_2 , x_2^* ) ) = ( x_1 , x_2^* ). \]
It is clear that $F$ is a bicharacter of the LCA(2) group $G\times G^*$.  The Heisenberg group $A(G)$ associated to $G$ is
 \[ A(G) = G\times G^* \times \Bbb T \]
 with the group
 law given by (\ref{2.2.2}).  We will write an automorphism of $G\times G^*$ by a $2\times 2$ matrix as in (\ref{2.2.3}), which acts on $G\times G^*$ from the
 right.  The group of automorphisms of $G\times G^*$ preserving the skew symmetric form $F(z_1 , z_2 ) F(z_2 , z_1 )^{-1}$ is denoted by $Sp ( G )$.

  It is easy to see that an automorphism $s$ of $A(G)$  that fixes the center $\Bbb T$ pointwisely  can be written as
   $ (w , t) s = ( w \sigma , f ( w ) t )$, where $\sigma $ is an automorphism of the abstract group $G\times G^*$ and $f : G\times G^* \to \Bbb T$ is a map
 satisfying (\ref{2.2.5}). We denote by
$B_0 ( G) $
 the group of $s$ as above with $ \sigma $ an automorphism of $G\times G^*$ as an LCA(2) group and $f$ a quadratic
 character of $G\times G^*$.    If $ x \mapsto 2x $ is an automorphism of LCA(2) group $G$, then
 $B_0 ( G) =  Sp( G) \ltimes ( G \times G^* ) $. The proof is the same as  that of the similar result in Section \ref{subsection2.2}.

 We have the following special elements in $B_0 ( G )$, similar to those in Section \ref{subsection2.2}.
    If $\alpha \in {\rm Aut} (  G )  $ is an automorphism, then
   \[      d_0 ( \alpha ) = \left(  \left( \begin{matrix} \alpha  & 0 \\ 0 & \alpha^{*-1} \end{matrix} \right) , 1 \right) \]
is in $B_0 (G) $.
 If $\gamma : G^* \to G $ is an isomorphism of $G^*$ onto $G$, then
 \[  d_0' ( \gamma ) =
 \left(  \left( \begin{matrix} 0  &  -\gamma^{*-1} \\ \gamma & 0  \end{matrix} \right) , f \right)
 \; \; \; {\rm with } \; \; \; f( x , x ^* ) = ( x , -x^* ) .\]
 If $h: G \to \Bbb T$ is a quadratic character with
 $  \rho : G \to G^*$ as in  (\ref{4.2}),
 then
\[  t_0 ( h ) =  \left(  \left( \begin{matrix} 1  &  \rho  \\ 0  & 1  \end{matrix} \right), h (x) \right)
 \]
is in $B_0 ( G)$. The map $h\mapsto t_0 ( h)$ is an embedding of the group of quadratic characters on $G$
 into $B_0 ( G ) $.
Similarly if $h: G^* \to \Bbb T$ is a quadratic character, then
 \[  t'_0 ( h ) =  \left(  \left( \begin{matrix} 1  &  0  \\ \rho  & 1  \end{matrix} \right) , h (x^* ) \right)
  \]
and $h\mapsto t_0' ( h)$ is an embedding of the group of
 quadratic characters on $G^*$ into $B_0 ( G ) $.

 The following is a generalization of Theorem \ref{thm2.2.1} and theirs proofs are exactly the same.

\begin{theorem}\label{thm4.2}
 Let $h(x)$ be a non-degenerate quadratic character of $G$  satisfying $ h ( x ) = h( -x ) $,  and $  \rho : G \to G^*$ be as in  (\ref{4.2}).
    Then there is a group homomorphism
 $SL_{2} ( {\Bbb Z} ) \to B_0 ( G ) $ given by
   \[   \left( \begin{matrix}   1 & 1 \\ 0 & 1 \end{matrix} \right) \mapsto t_0 ( h ),  \; \; \;  \; \;  \left( \begin{matrix}   0 & -1 \\ 1 & 0 \end{matrix} \right) \mapsto d_0' ( \rho^{-1} ).  \]
\end{theorem}

  For a fixed $ V_0 \in {\cal U} ( G )$,  we defined  ${\cal S}_{V_0} ( G ) $, the Bruhat-Schwartz space relative to $V_0$ in Section \ref{ss4.3}.
   We will also denote it by ${\cal S} ( G )$ if the choice of $V_0$ is given.
 We now define an action of  the Heisenberg group $A ( G )$  on ${\cal S} ( G )$.
  Let  $ f \in  {\cal S} ( G )$ with $(V_1 , V_2)$-component given by
 \begin{equation}\label{4.3}    f_{V_1 , V_2 } \otimes \mu_{V_1, V_2}   \in  {\cal S}_{V_0} (V_1 , V_2 ) = {\cal S} ( V_2/ V_1 ) \otimes \mu ( V_1 , V_0 ).  \end{equation}
  For $a \in G$,  suppose that $ a \in V$ for some $V\in {\cal U} ( G )$,  then
  \[ {\cal N}_a := \{  ( V_1  , V_2 ) \in {\cal U}^2  \, | \,  V_1 \subset V \subset V_2  \} \]
 is a net. We ignore the dependence of $\cal N_a$ on $V$ since we shall only use the property that $\cal N_a$ is a net. For $(V_1 , V_2 ) \in {\cal N }_a $,
    consider
 \begin{equation}\label{4.4}        f_{ V_1 , V_2 } ( x+ a ) \otimes \mu_{V_1 , V_2 } \in  {\cal S}_{V_0} ( V_1 , V_2 ). \end{equation}
  It is straightforward to see that the family (\ref{4.4})  for  $(V_1 , V_2 ) \in {\cal N}_a $
  is compatible in the sense of (\ref{3c}),  so it can be extended to a family for all $( V_1 , V_2 )\in {\cal U}^2 $ with $V_1 \subset V_2 $
   by (\ref{3ext}).
      We get an element $ \pi (a ) f \in {\cal S} ( G)$. Using the fact that ${\cal N}_{a_1} \cap {\cal N}_{a_2}$ is also a net,
       we can prove
    $ \pi (a_1 + a_2 ) = \pi ( a_1 ) \pi (a_2 ). $

 For $a^* \in G^*$,  suppose that $a^* \in V^{\perp} $ for some $V\in {\cal U}$.  Let
 \[   {\cal N}_{a^*} = \{  ( V_1  , V_2 ) \in {\cal U}^2  \, | \,  V_1 \subset  V  \subset V_2  \} ,\]
which is clearly a net of $G$.  For  $(V_1 , V_2 ) \in {\cal N }_{a^*} $, we consider
\begin{equation}\label{4.5}      ( x , a^* )    f_{ V_1 , V_2 } ( x ) \otimes \mu_{V_1 ,V_2 } \in  {\cal S}_{V_0} ( V_1 , V_2 ). \end{equation}
 The family  (\ref{4.5})   for $(V_1 , V_2 ) \in {\cal N}_{a^*} $
  is compatible in the sense of (\ref{3c}), so it can be extended to a family for all $( V_1 , V_2 )\in {\cal U}^2 $ with $V_1 \subset V_2$
   by (\ref{3ext}).
  We get an element $ \pi (a^* ) h \in {\cal S} ( G)$.
It is clear that $ \pi (a_1^* + a_2^* ) = \pi ( a_1^* ) \pi (a_2^* ) $.
 Finally for $( a , a^*, \lambda ) \in A( G)$, we define
 $ \pi ( a , a^*, \lambda  ) = \lambda \, \pi  ( a^* ) \pi ( a )$.

 \begin{lemma}\label{lemma4.3} The above gives a representation of $A ( G ) $  on ${\cal S} ( G ).$
\end{lemma}

If $ V \subset U$ in ${\cal U}(G)$, then $U^{\perp}\subset V^{\perp}$ in ${\cal U} (G^*)$.
 The LCA groups $U / V $ and $ V^{\perp} / U^{\perp}$ are dual LCA groups. We have the Heisenberg group
 \[  A( U / V )  =   U / V  \times V^{\perp} /  U^{\perp} \times \Bbb T \]
as in Section \ref{subsection2.2}.   The group $A( U / V) $  is the quotient of $ U \times V^{\perp} \times \Bbb T \subset A(G)$ by $ V\times U^{\perp}$:
 \[     \pi:     U \times V^{\perp} \times \Bbb T  \to    A( U / V ) , \; \;  ( x , x^* , t ) \mapsto  (  x+ V,     x^* + U^{\perp} , t ) .\]
The group  $ A( U / V )$ acts on ${\cal S} ( U / V)$ as in Section \ref{ss3.3}.    Fixing a nonzero measure $ \mu \in  \mu ( V , V_0 )$,  we have a map
  ${\cal S}_{V_0 } ( G ) \to {\cal S} ( U / V) $ which sends  $ f \in {\cal S}_{V_0 } ( G )$
  with $(V_1 , V_2)$-components as in (\ref{4.3})
  to $ f_{V, U} \otimes  \mu_{V, U} /  \mu  $. It is clear that we have the commutative diagram
\begin{equation}\label{4.6}
\xymatrix{
  ( U\times V^{\perp} \times \Bbb T) \times    {\cal S}_{V_0} ( G ) \ar[r] \ar[d]  &   {\cal S}_{V_0} ( G ) \ar[d] \\
    A(U/ V) \times {\cal S} ( U / V  )   \ar[r]   &  {\cal S} ( U / V  )
}
\end{equation}

Similarly $ A ( G )$ acts on $ {\cal S}_{V_0^{\perp}} ( G^* ) = {\cal S} ( G^* )$ as follows.
 Let $ f'\in  {\cal S} ( G^* )$ with $(W_1 , W_2 )$-component
 \[  f'_{ W_1 , W_2 }\otimes \mu_{W_1 , W_2 } \in  {\cal S}_{V_0^{\perp} }( W_1 , W_2 ) =     {\cal S} ( W_2/ W_1 ) \otimes \mu ( W_1 , V_0^{\perp} ).\]
  For $ a \in G$,
  \[  {\cal N}'_a := \{ ( W_1 , W_2 ) \in {\cal U} ( G^* )^2  \; | \;   W_1 \subset W_2 ,  ( W_1 , a ) =1 \} \]
  is a net for $G^*$.
 Then  $\pi  (a ) f' $  is the unique element in ${\cal S} ( G^* ) $ such that for $(W_1 , W_2 ) \in  {\cal N}'_a$,
   the $(W_1 , W_2 )$-component of  $ \pi  (a ) f' $ is given by
  \[    f'_{ W_1 , W_2 } ( x^* ) ( a , - x^* )  \otimes \mu_{W_1 , W_2 } . \]
 For $ a^* \in G^* $,
  \[  {\cal N}'_{a^*}: = \{ ( W_1 , W_2 ) \in {\cal U} ( G^* )^2  \; | \;   W_1 \subset W_2 ,  a^* \in W_2  \} \]
  is a net for $G^*$. Then $\pi ( a^* ) f'$ is the unique element in ${\cal S} ( G^* ) $ such that for $(W_1 , W_2 ) \in  {\cal N}'_{a^*}$,
   the $(W_1 , W_2 )$-component of  $ \pi  (a^* ) f' $ is given by
  \[    f'_{ W_1 , W_2 } ( x^* + a^*  )   \otimes \mu_{W_1 , W_2 } . \]
    The central element  $ \lambda \in \Bbb T$ acts as scalar multiplication by $\lambda $.

 \begin{lemma}\label{lemma4.4}  The Fourier transform $ F:       {\cal S} ( G ) \to {\cal S} ( G^* )$ is an intertwining operator between representations of $A ( G )$.
\end{lemma}

 Recall that  we defined an action of $\widetilde{\rm Aut} ( G ) $ on ${\cal S} ( G)$  in Section \ref{ss4.3} (see (\ref{3act})) .
   It is easy to check that
  for every $  \alpha' = ( \alpha , \lambda ) \in \widetilde {\rm Aut } ( G )$,   $ g \in A ( G ) $,
  \begin{equation}\label{4.7}
    \pi ( \alpha' )^{-1}   \pi  ( g )     \pi ( \alpha' ) =  \pi  ( g d_0 ( \alpha ) )   .\end{equation}
Let $ \alpha : G \to G' $ be an isomorphism in LCA(2). In order to have an induced isomorphism  ${\cal S}_{V_0'} ( G' ) \to
  {\cal S}_{V_0} ( G ) $, we take $\lambda \in \mu (  V_0' \alpha^{-1} , V_0 ) $, $\lambda \ne 0 $.   The pair $( \alpha , \lambda ) $
   gives an isomorphism
 \[  T ( \alpha , \lambda ) : {\cal S}_{V_0'} ( G' ) \to    {\cal S}_{V_0} ( G ) \]
 as follows.
      Let $ f \in {\cal S}_{V_0'} ( G' ) $, whose $(W_1 , W_2 )$-component
   is  $  f_{W_1 , W_2 } ( x ) \otimes \mu_{W_1, W_2 } \in {\cal S} ( W_2 / W_1 ) \otimes \mu ( W_1 , V_0' ) $.
Define $ T (  \alpha , \lambda ) f \in  {\cal S}_{V_0} ( G ) $ as the element whose $(V_1 , V_2 )$-component is
    \begin{equation}\label{4.8}      f_{V_1\alpha , V_2\alpha  } ( x\alpha  ) \otimes  ( \mu_{V_1\alpha , V_2 \alpha  } \alpha^{-1} ) \lambda,
    \end{equation}
where $ \mu_{V_1\alpha , V_2 \alpha  } \alpha^{-1} $ is the image of $\mu_{V_1\alpha , V_2 \alpha  }$  under
the LCA isomorphism $ \alpha^{-1} :    V_2 \alpha  / V_1 \alpha   \to V_2 / V_1 $.

\subsection{Weil index for a non-degenerate quadratic character on LCA(2) group}

 Let $h$ be a quadratic character of $G$ and $\rho :  G \to G^*$ be the associated morphism as in (\ref{4.2}).
 We define an automorphism $ \pi ( t_0 ( h )  ) $ of ${\cal S} ( G ) $
  as follows.  Let
  \[ {\cal N}_h = \{  ( V , U ) \, | \,    U , V \in {\cal U},  V \subset U,  h ( u + v ) = h ( u ) \; \; {\rm for \; all } \; u \in U , v \in V \} .\]
 By Lemma \ref{lemma4.1}, $ {\cal N}_h$ is a net.
Let $f \in {\cal S} ( G )$ with $ (V_1 , V_2 )$-component given by (\ref{4.3}).
 For $ (V , U) \in {\cal N}_h $, the quadratic character $ h ( x ) $ gives a quadratic character on $ U / V$, which we still denote by $h( x ) $.
 Define $\pi ( t_0 ( h )  ) f $ to be the element with $(V, U)$-component for $(V, U)\in {\cal N}_h$ given by
    \begin{equation}\label{4.9}
                h ( x )   f_{ V , U } ( x) \otimes \mu_{V , U } .
                \end{equation}
  It is clear that the family (\ref{4.9}) for $(V, U)\in {\cal N}_h$ is compatible, so our definition of  $\pi ( t_0 ( h )  ) f $ is well-defined.
It is clear that this defines  an action of the group of quadratic characters on  $ {\cal S} ( G ) $.

 For an automorphism $\alpha \in {\rm Aut} ( G ) $, choose $\lambda \in \mu ( V_0 \alpha^{-1} , V_0 ) $, $\lambda \ne 0 $.
 We define $ \pi ( d_0 ( \alpha ) , \lambda ) $ as the action of $ (  \alpha , \lambda ) \in \widetilde{\rm Aut} ( G ) $
  on $ {\cal S} ( G ) $ (see Section \ref{ss4.3}, (\ref{3act})).

 Let $ \gamma : G^* \to G$ be an isomorphism, and $ \lambda  \in \mu ( V_0^{\perp} \gamma^{*} , V_0 )$.    We define an isomorphism
 \[  \pi  ( d_0' ( \gamma )  , \lambda ) :
  {\cal S} ( G ) \to  {\cal S} ( G ) \]
   as the composition
   \begin{equation}\label{4.10}   {\cal S} ( G )  \stackrel { F }  { \longrightarrow }  {\cal S} ( G^* )   \stackrel { T( -{\gamma^*}^{-1} , \lambda )  }  { \longrightarrow }   {\cal S} ( G ).
   \end{equation}
  This definition is analogous to that of $\pi ( d_0' ( \gamma ) ) $ in Section \ref{subsection2.2}.

  For $f \in {\cal S}_{V_0 } ( G )$ with $(V_1 , V_2 )$-component
\[     f_{V_1 , V_2 } ( x ) \otimes \mu_{V_1 , V_2 } ,\]
 write
  $ \mu_{V_1 , V_2 } \in \mu ( V_1 , V_0 ) $ as
\[     \mu_{V_1 , V_2 } =  \mu_{V_1 , V_2 }'  \otimes \mu_{V_1 , V_2 }'' , \; \; \;  \mu_{V_1 , V_2 }' \in \mu ( V_1 , V_2 ), \; \;  \mu_{V_1 , V_2 }''\in \mu ( V_2 , V_0)=  \mu ( V_2^{\perp} , V_0^{\perp}) .\]
  Then the $(V_2^{\perp} \gamma^*  , V_1^{\perp} \gamma^*)$-component of $  \pi ( d_0' ( \gamma )  , \lambda ) \phi $ is given by
\begin{equation}\label{4.11}
   \int_{V_2/V_1} f_{V_1 , V_2 }  ( t ) \la t , -x^* {\gamma^*}^{-1} \ra  d\mu_{V_1 , V_2}' (t)  \otimes (\mu_{V_1 , V_2}'' \gamma^*)\lambda  .
\end{equation}

We will write $ \pi ( d_0 ( \alpha ) , \lambda )$ and $ \pi  ( d_0' ( \gamma )  , \lambda )$ as $ \pi ( d_0 ( \alpha ) )$ and $ \pi  ( d_0' ( \gamma )   )$ when the choice of $\lambda $ is given.
Let $ g \in B_0 ( G ) $ be one of $ d_0 ( \alpha ) ,  t_0 ( h ) ,  d_0' ( \gamma ) $, and $a \in A( G ) $. Then
\[          \pi ( g )^{-1}  \pi ( a ) \pi ( g ) = \pi ( a \cdot g ).\]

Let $h$ be a non-degenerate quadratic character of $G$ and
 $\rho: G\to G^\ast$ be the associated isomorphism.  Fix a choice of $\lambda \in \mu ( V_0^{\perp} \rho^{-1} ,  V_0 )$. $\lambda \ne 0 $.
 So we have operators $\pi ( t_0 ( h )  )$ and $  \pi (  d_0' ( \rho^{-1} ) ) $.
 We will prove that they satisfy the relation similar to (\ref{2.3rel}) and give rise to a projective action of $ SL_{2} ( {\Bbb Z}) $
 on ${\cal S} (G ) $.

 For a pair $(U^\perp\rho^{-1}, U)$ as in Lemma \ref{lemma4.1}, we write $ U' $ for $U^\perp\rho^{-1}$.  Property   (2) in Lemmma \ref{lemma4.1} implies that $ h $ descends to a non-degenerate quadratic character
  $\bar h $ on $ \bar U = U / U' $, and $\rho$ induces an isomorphism $ \bar \rho  : \bar  U  \to {\bar U}^* =  U'^{\perp} / U^{\perp } $. As in Section \ref{subsection2.2}, we have operators
  \[  \pi ( t ( \bar h ) ) ,  \pi ( d_0' ( {\bar \rho}^{-1} )): \; \; {\cal S} ( \bar U ) \to {\cal S} ( \bar U ) \]
 that satisfy the relations  (\ref{4power}) and
 (\ref{2.3rel}). Therefore they generate a projective representation of $SL_2 ( {\Bbb Z}  ) $ on  ${\cal S} ( \bar U )$.
  Fix a $\mu \in  \mu ( U' , V_0 ) , \mu \ne 0 $.      It is easy to check that the following diagrams commute
  \begin{equation}\label{4.12}
\xymatrix{
     {\cal S}_{V_0} ( G ) \ar[r]^{ \pi ( t_0 ( h )) }    \ar[d]  &   {\cal S}_{V_0} ( G ) \ar[d] \\
    {\cal S} ( \bar U  )   \ar[r]^{ \pi ( t_0 (\bar  h )) }   &  {\cal S} ( \bar U   )
}    \; \; \; \; \; \;
\xymatrix{
     {\cal S}_{V_0} ( G ) \ar[r]^{ \pi ( d_0' ( \rho^{-1} )  ) }  \ar[d]  &   {\cal S}_{V_0} ( G ) \ar[d] \\
    {\cal S} ( \bar U  )   \ar[r]^{ \pi ( d_0' ( {\bar \rho}^{-1} )   ) }  &  {\cal S} ( \bar U   )
}
\end{equation}
   where all the four vertical arrows $ {\cal S}_{V_0} ( G ) \to   {\cal S} ( \bar U  ) $ send
  $ f $ with $ ( U',  U) $-component $ f_{ U',  U } \otimes \mu_{ U' , U }\in {\cal S} ( U / U' ) \otimes \mu ( U ' , V_0 )$
  to $  f_{ U' , U }  \mu_{ U' , U } / \mu $.
  From these diagrams together with (\ref{4power}) and (\ref{2.3rel}), it follows that  there exist nonzero scalars $c(U)$ and $d(U)$ such that
 for every $ f \in {\cal S}_{V_0} ( G ) $
 and $ ( U' , U) $ as above,
 \[ (U' , U){\rm -component \; of}  \;   \left( \pi ( t_0 ( h ) ) \pi ( d_0' ( \rho^{-1} ) ) \right)^3 f  = c(U) \cdot   (U' , U){\rm-component \; of} \; \pi (  d_0' ( \rho^{-1} ))^{2}f, \]
 \[  (U' , U){\rm-component \; of}  \;   \pi ( d_0' ( \rho^{-1} ) )^4 f   = d(U) \cdot   (U' , U){\rm-component \; of} \; f.\]
 Since the $(U' , U)$-components of $\left( \pi ( t_0 ( h ) ) \pi ( d_0' ( \rho^{-1} ) ) \right)^3 f $ are compatible,
  all the $c(U)$'s are equal. Similary all the $d(U)$'s are equal.  This proves the following

\begin{theorem}\label{thm4.5}  Let $h$ be a non-degenerate quadratic character of $G$ and
 $\rho: G\to G^\ast$ is the associated isomorphism.  Then $ \pi ( t_0 ( h ) )$ and
$\pi ( d_0' ( \rho^{-1} ) )$ generate a projective representation of   $SL_2 ( {\Bbb Z} ) $ on ${\cal S} ( G ) $.
\end{theorem}

\begin{defn}\label{def} {\rm Let $h$ be a non-degenerate quadratic character on an LCA(2) group $G$ and let $\rho : G\to G^*$ be
 the associated isomorphism. Let $U\in {\cal U } (  G ) $ satisfying the properties in Lemma \ref{lemma4.1}, and $\bar h$ be the non-degenerate quadratic character on $\bar U :=  U /  U^{\perp} \rho^{-1}$
  induced from $h$. Define the {\bf Weil index} $\gamma ( h)$  of $h$ as $ \gamma ( \bar h )$. }
 \end{defn}

The existence of $U$ in Definition \ref{def} follows from Lemma \ref{lemma4.1}.  To see that $\gamma ( h )$
is independent of the choice of $U$, we use the fact that
 \[      ( \pi ( t_0 ( h ) ) \pi ( d_0' ( \rho^{-1} ) )^3 = C  \pi ( d_0' ( \rho^{-1} ) )^2  \]
 for some nonzero scalar $C$, and $ \gamma ( h ) = C / | C| $.

 \section{ Product formula of Weil index } \label{s6}

In this section we generalize the product formula for the Weil index and give several examples.
We first prove a generalization of Corollary \ref{cor2.6}.

\begin{theorem}\label{thm5.1}
Let $G$ be an LCA(2) group with a non-degenerate quadratic character $ h $ satisfying $h(x)=h(-x)$, and $\rho : G\to G^*$
be the associated isomorphism. If there is a closed subgroup $H$ of $G$ (i.e., $H$ is the image of an admissible monic into $G$)
 such that $ H^{\perp}  \rho^{-1} = H $ and $h(H^\perp\rho^{-1})=h(H)=1$, then $\gamma ( h ) =1 $.
\end{theorem}

\begin{proof}   Let ${\cal B} $ be a basis of $G$ such that $ \{ S \cap H \ | \ S\in {\cal B} \}$
 is a basis of $H$ and for any $ S_1 \subset S_2$ in ${\cal B} $ the map
$  S_2 \cap H /   S_1\cap H \to S_2 / S_1 $ is a closed embedding.
 The  method  in the proof of Lemma \ref{lemma4.1} can be used to prove that
 there exists  $ U \in {\cal B}$ satisfying the properties in the lemma.
 We have the non-degenerate quadratic character  $ \bar h $ on $\bar U := U /  U^{\perp} \rho^{-1}$
  induced from $h$. The group
  $ \bar H := U \cap H /  ( U^{\perp} \rho^{-1} \cap H ) $ is a closed subgroup of $\bar U $.
  It is easy to see that  $  {\bar H}^{\perp}{ \bar  \rho}^{-1}  = \bar H  $, and $\bar{h}( {\bar H}^{\perp}{ \bar  \rho}^{-1} )=\bar{h}(\bar H)=1$.
  Therefore $ \gamma ( h ) = \gamma ( \bar h ) = 1 $.
\end{proof}

\subsection{Product formula for a curve over  local field}\label{ss6.1}
Let $k$ be a local field, and $C$ be an irreducible regular projective curve over $k$ with function field $k(C)$.
The adele ring ${\Bbb A}_C$ of $C$ is by definition the restricted direct product over all the closed points of $C$:
\[   {\Bbb A}_C :=  {\prod}'_v F_v  \]
where $ F_v = \textrm{Frac } \widehat {\cal O}_v  $ is a two-dimensional local field.
An element $(a_v ) \in    \prod_v F_v   $ is in  $ \prod_v' F_v $ if and only if
$a_v \in  \widehat{\cal O}_v $ for almost all $v$.  For a divisor $ D =\sum  m_v [ v ] $ on $C$,
 set
\[  {\Bbb A}_C ( D ) = \{ ( a_v ) \in {\Bbb A} \ | \  \textrm{ord}_v ( a_v ) + m_v \geq 0    \ {\rm for \ all \ } v\in C  \} . \]
Then we have a filtration $ {\Bbb A}_C ( D ) $ indexed by the divisor group $\textrm{Div} ( C)  $.
If $D_1 < D_2 $, i.e.,  $ D_2 - D_1 = \sum_{i=1}^s   m_i [ v_i ] $ with $ m_i \geq 0 $,
then $ \Bbb A_C( D_1 ) \subset \Bbb A_C ( D_2 ) $ and
$ \Bbb  A_C( D_2 ) / \Bbb A_C( D_1 ) $ is isomorphic to a direct product
\[  \prod_{i=1}^s k_{v_i}^{m_i} \]
where $ k_{v_i} $ is the residue field at $v_i$.   Each $ k_{v_i} $ a finite
extension of the local field $k$, hence also a local field.
Therefore $   \Bbb A_C( D_2 ) / \Bbb A_C( D_1 ) $ is an LCA.
The family $\{ {\Bbb A}_C(D) \ | \ D\in \textrm{Div}(C)\}$ satisfies all the conditions in Lemma \ref{lemma3.1},  hence ${\Bbb A}_C$ has an LCA(2) structure.

 Let  us fix a non-trivial additive character $\psi $ of $k$, which gives an additive character $\psi_v$ on each
 $k_v$ by the composition $ k_{v}  \stackrel{\textrm{Tr}}{\longrightarrow} k \stackrel{\psi}{\longrightarrow} \Bbb T $.
 We also fix a nonzero meromorphic differential $\omega $ on $C$,
which gives a pairing $ F_v \times F_v \to  k $ by
 \begin{equation}\label{5.1}
   ( f_v , g_v ) = \textrm{Tr}_{k_v/k}( {\res}_v ( f_v g_v \omega ) ) .
   \end{equation}
   The residue   ${\res}_v ( f_v g_v \omega ) )$ can be computed as follows.
If we choose a local parameter $t_v$ at $v$, then the local differential $f_v g_v \omega$ can be written as
 \[  \sum_{i= -\infty}^{\infty}  c_i t_v^i dt_v,  \]
 and $ {\res}_v ( f g \omega ) ) = c_{-1} \in k_v $. It is clearly independent of the choice of $t_v$.

 The local pairings (\ref{5.1}) give a global one
 \begin{equation}\label{5.2}
    {\Bbb A}_C \times {\Bbb A}_C \to k  , \ \  ( ( f_v ) , (g_v ) ) = \sum_v ( f_v , g_v ).
    \end{equation}
 We  identify  the dual vector space of  $k(C)^n$ with itself by the dot product on  $k(C)^n$:
 \[  ( f_1 , \dots , f_n ) \cdot  ( g_1 , \dots , g_n ) = \sum f_i g_i . \]
Then a quadratic form $q$ on $k(C)^n$ induces a linear map $\rho: k(C)^n\to k(C)^n$ by
 \[  q ( f + g ) = q (f ) + q ( g ) +  ( f \rho ) \cdot g . \]
Suppose that $q$ is non-degenerate, so that $\rho $ is an isomorphism.
By scalar extension $q$ gives a non-degenerate quadratic from on $F_v^n$, which we again denote by $q$. Then it gives a non-degenerate quadratic character of the LAC(2) group $ F_v^n $ by
\[    h_{  v} ( f_v ) :=  \psi_v (   {\res}_v ( q(f_v )  \omega ) ) .\]
 We denote  the Weil index of  $h_{ v}$ by $\gamma_v ( q ) $, which depends on the choices
  of $\psi$ and $ \omega $.

\begin{theorem}\label{5.2}   $\gamma_v ( q ) = 1 $ for almost all closed points $v\in C$, and
\[  \prod_v \gamma_v ( q ) = 1.
\]
\end{theorem}

This result follows from Theorem \ref{thm5.1} and the duality result that
$ (f_v ) \in {\Bbb A}_C$ is in $k( C) $ iff $ ( ( f_v ) , g ) ) = 0 $ for all $g \in  k( C) $,
 which implies that $\rho$ gives an isomorphism $k(C)\stackrel{\sim}{\rightarrow} k(C)^\perp$.
Applying the theorem for $n=1$ and the quadratic form on $ k (C)$ given by $q ( f ) =  \frac 1 2 f^2  $ (assuming that 2 is invertible in $k$),
 we have

\begin{cor}\label{cor5.3} Assume that $k$ is not of characteristic 2. Let $\omega $ be a nonzero meromorphic differential on $C$.
 For each closed point $ v \in C$, take a local parameter $t_v $ and expand $\omega $ at $v$ as
 \[  \sum_{i=0}^{\infty}  c_{i, v} t_v^{{\rm ord}_v \omega + i  } d t_v.      \]
Define $ \gamma_v ( \omega) = 1 $ if $ {\rm ord}_v \omega  $ is even, and
  $ \gamma_v ( \omega ) = \gamma (  \psi_v(\frac{1}{2}c_{0,v}x^2) )  $ otherwise (here $\gamma$ is the usual Weil index on the local
field $k_v$). Then $\gamma_v(\omega)=1$ for almost all $v$ and
  \[ \prod_v \gamma_v ( \omega ) =  1.\]
 \end{cor}

\begin{proof}
Take $q(f)=\frac{1}{2}f^2$. If ord$_v\omega=2d$ is even, then $h_v$ induces an isomorphism $\rho_v: t_v^{-d}\widehat{O}_v
\stackrel{\sim}{\rightarrow}  (t_v^{-d}\widehat{O}_v)^\perp$, which implies that $\gamma_v(q)=1$ by Theorem \ref{thm5.1}.
If ord$_v\omega=2d+1$ is odd, then $(t_v^{-d-1}\widehat{O}_v)^\perp\rho_v^{-1}=t_v^{-d}\widehat{O}_v\subset t_v^{-d-1}\widehat{O}_v $, hence $h_v$ induces a quadratic character
\[
\bar{h}_v(x)=\psi_v(\frac{1}{2}c_{0,v}x^2)
\]
on the quotient
\[
t_v^{-d-1}\widehat{O}_v/t_v^{-d}\widehat{O}_v\stackrel{\sim}{\rightarrow} k_v,\quad xt_v^{-d-1}+t_v^{-d}\widehat{O}_v\mapsto x,\quad x\in k_v.
\] Note that $\gamma (  \psi_v(\frac{1}{2}c_{0,v}x^2) )$ does not depend on the choice of  $t_v$. In fact if   $c_{0,v}'$ is the coefficient in the expansion of $\omega$ using another local parameter $t_v'$, then $c_{0,v}'/c_{0,v}$ is a square. We see that  $\gamma_v(\omega)$ defined in the corollary is equal to the local Weil index $\gamma_v(q)$, and the corollary follows from Theorem \ref{5.2}.
\end{proof}


\subsection{Product formula on an arithmetic surface}

We first recall the setting and main result of \cite{L}, where more details can be found. Let $\mathcal O_K$ be a Dedekind domain of characteristic zero with finite residue fields. Let $X$ be a two-dimensional, normal scheme, flat and projective over $S=\textrm{Spec }\mathcal O_K$, whose generic fibre is one-dimensional and irreducible. We consider a semi-local situation, namely we fix a closed point $x\in X$ lying over a finite place $s$ of $\mathcal O_K$, and take various formal curves $y$ containing $x$. Let $K(X)_x=\textrm{Frac~}\widehat{\mathcal O}_{X, x}$ and $K_s=\textrm {Frac }\widehat{\mathcal O}_{K,s}$. We may define the adele ring ${\Bbb A}_{X,x}$ at $x$ and the adelic space of continuous relative differential forms $\Omega^{\rm cts}_{{\Bbb A}_{X,x}/K_s}$ as certain restricted products, between which there is a canonical residue pairing
 \[
 {\Bbb A}_{X,x}\times \Omega^{\rm cts}_{{\Bbb A}_{X,x}/K_s}\to K_s, \quad (f, \omega)\mapsto \textrm{Res}_x(f\omega):=\sum_{y\subset^f X, y\ni x}\textrm{Res}_{x,y}(f_y\omega_y)
 \]
 where the last sum is taken over all formal curves $y\subset^f X$ containing $x$. The main result of \cite{L}  states that,
 $K(X)_x$ and $ \Omega^{\rm cts}_{K(X)_x/K_s}$ are exactly the mutual annihilators under the residue paring at $x$. We refer the readers to \cite{L} for the precise definitions of some notions above.

The adele ring $\Bbb A_{X,x}$ can be filtered by formal divisors $D=\sum_{y\subset^f X, y\ni x}m_y[y]$ in a similar way as in Section \ref{ss6.1}, which gives an LCA(2) structure on
$\Bbb A_{X,x}$.  Take a non-degenerate quadratic form $q$ of $K(X)_x^n$ and a non-trivial additive character $\psi$ of $K_s$. Let $\omega\in \Omega^{\rm cts}_{K(X)_x/K_s}$ be a nonzero rational continuous differential form. Define the global quadratic character
\[
h(f)=\psi(\textrm{Res}_x(q(f)\omega))
\]
of ${\Bbb A}_{X,x}^n$, and the local quadratic character
\[
h_y(f)=\psi(\textrm{Res}_{x,y}(q(f)\omega))
\]
of $K^n_{x,y}$, where $K_{x,y}$ is the two-dimensional local field associated to the datum $x\in y\subset^f X$. Let $\gamma_y(q)$ be the Weil index of $h_y$, which depends on $\psi$ and $\omega$. The main result of \cite{L} implies that $h$ induces an isomorphism
$\rho: K(X)_x^n\stackrel{\sim}{\rightarrow}  (K(X)_x^n)^\perp$. The explicit description of the residue map Res$_{x,y}$ implies that $h_y$ induces an isomorphism
$\rho_y: \cal O_{x,y}^n\stackrel{\sim}{\rightarrow}  (\cal O_{x,y})^\perp$ for almost all  formal curves $y$ containing $x$, where $\cal O_{x,y}=\cal O_{K_{x,y}}$. By Theorem
\ref{thm5.1}, we obtain

\begin{theorem}\label{thm6.4}
$\gamma_y(q)=1$ for almost all formal curves $y$ containing $x$, and
\[
\prod_{y\subset^f X,y\ni x}\gamma_y(q)=1.
\]

\end{theorem}

Let us give an example. Let $X=\Bbb P^1_{\cal O_K}\supset \Bbb A^1_{\cal O_K}=\textrm{Spec }\cal O_K[t]$ and $x=(\frak{p}, t)$, where  $\frak{p}$ is a prime ideal of $\cal O_K$ with residue field $k_\frak{p}$. In other words, $x$ is the zero point of the fibre over $\frak{p}$. We have $\widehat{\cal O}_{X,x}=\cal O_\frak{p}[[t]]$, where $\cal O_\frak{p}$ is the ring of integers of the local field $K_\frak{p}$. By abuse of notation, we also denote by $y$ the height one prime ideal of $\cal O_\frak{p}[[t]]$ corresponding to a formal curve $y$ containing $x$. By Weierstrass' preparation theorem, $y$ is generated by either $\frak{p}$ or an irreducible distinguished polynomial (i.e., of the form $P(t)=t^l+a_1t^{l-1}+\cdots +a_l$ with $a_i\in\frak{p}$). We denote $y$ by $y_\frak{p}$ or $y_P$ for these two cases respectively. Then
\begin{align*}
& K_{x,y_\frak{p}}=K_\frak{p}\{\{t\}\}:=\left\{\sum_{i\in\Bbb Z}a_i t^i\ |\  a_i\textrm{ is bounded and }\lim_{i\to \infty}a_{-i}=0\right\},\\
& K_{x,y_P}=K_P((t_P)),
\end{align*}
where $K_P:=\cal O_K[t]/(P(t))$ and $t_P$ is the local parameter of $y_P$ given by $P(t)$.

Recall that (see e.g. \cite{M} and \cite{L}) the local continuous relative differential forms and the residue maps are given by
\begin{align*}
&\textrm{Res}_{x, y_\frak{p}}: \Omega^{\rm cts}_{K_{x, y_\frak{p}}/K_\frak{p}}=K_{x, y_\frak{p}}dt\to K_\frak{p},\quad \sum_i a_i t^idt\mapsto -a_{-1},\\
&\textrm{Res}_{x, y_P}: \Omega^{\rm cts}_{K_{x, y_P}/K_\frak{p}}=K_{x, y_P}dt_P\to K_\frak{p},\quad \sum_i a_i t_P^idt_P\mapsto \textrm{Tr}_{K_P/K} a_{-1},
\end{align*}
which are independent of the choice of local parameters.

Let $\psi_P$ be the character of $K_P$ given by
\[
\psi_P(f)=\psi(\textrm{Tr}_{K_P/K_\frak{p}}f).
\]
Assume that $\psi$ has conductor $\frak{p}^{c_\psi}$, $c_\psi\in\Bbb Z$, and choose a local parameter $\varpi_\frak{p}$ of $\frak{p}$.  Then $\psi$ induces an additive character
\[
\psi_0(a):=\psi(\tilde{a})
\]
of $k_\frak{p}$,
where $\tilde{a}$ is any lift of $a$ in $\frak{p}^{c_\psi-1}$ under the map  $
\frak{p}^{c_\psi-1}\to k_\frak{p}$, $\tilde{a}\mapsto \varpi_\frak{p}^{1-c_\psi}\tilde{a}\textrm{ mod }\frak{p}$.
Let $\psi_\frak{p}$ be the character of $k_\frak{p}((t))$ given by
\[
\psi_\frak{p}(\sum_i a_i t^i)=\psi_0(-a_{-1}).
\]
Define a map
\[
K_\frak{p}\{\{t\}\}^\times \to k_\frak{p}((t))^\times,\quad f\mapsto \bar{f}:=\varpi_\frak{p}^{-\textrm{ord}_{y_\frak{p}}f} f \textrm{ mod } y_\frak{p},
\]
and put $\bar{0}=0$. Then obviously $\overline{fg}=\bar{f}\cdot\bar{g}$.

Applying Theorem \ref{thm6.4} for $n=1$ and $q(f)=\frac{1}{2}f^2$, we obtain the following

\begin{cor}\label{cor6.5}
Assume that $k_\frak{p}$ has odd characteristic. Let $\omega\in\Omega^{\rm cts}_{F_\frak{p}/K_\frak{p}}$ be a nonzero continuous relative differential, where $F_\frak{p}={\rm Frac}~{\cal O}_\frak{p}[[t]]$. Assume that $\omega$ has the expansion
\[
\sum^\infty_{i=0}c_{i,P}t_P^{{\rm ord}_{y_P}\omega+i}dt_P
\]
at $y_P$, and that $\omega=c_\frak{p}(t)dt$ at $y_\frak{p}$, where $c_\frak{p}\in K_\frak{p}\{\{t\}\}$. Write ${\rm ord}_{y_\frak{p}}\omega:={\rm ord}_{y_\frak{p}}c_\frak{p}$. Define
\[
\gamma_P(\omega)=\left\{\begin{array}{ll} 1 & \textrm{if }{\rm ord}_{y_P}\omega\textrm { is even}\\
\gamma(\psi_P(\frac{1}{2}c_{0,P}a^2)) & \textrm{otherwise,}\end{array}\right.
\]
and
\[
\gamma_\frak{p}(\omega)=\left\{\begin{array}{ll} 1 & \textrm{if }{\rm ord}_{y_\frak{p}}\omega-c_\psi\textrm { is even}\\
\gamma(\psi_\frak{p}(\frac{1}{2}\bar{c}_\frak{p}a^2)) & \textrm{otherwise.}\end{array}\right.
\]
Then $\gamma_P(\omega)=1$ for almost all $P$ and
\[
\gamma_\frak{p}(\omega)\prod_P\gamma_P(\omega)=1,
\]
where the product is taken over all irreducible distinguished polynomials in $\cal O_\frak{p}[t]$.
\end{cor}

\begin{proof}
By Theorem \ref{thm6.4}, it suffices to show that $\gamma_P(\omega)=\gamma_{y_P}(q)$ and $\gamma_\frak{p}(\omega)=\gamma_{y_\frak{p}}(q)$, where $q(f)=\frac{1}{2}f^2$. The proof of the first equation is similar to that of Corollary \ref{cor5.3}, and we shall only prove the second one. If ord$_{y_\frak{p}}\omega=c_\psi+2d$, then $h_{y_\frak{p}}$ induces an isomorphism $\rho_{y_\frak{p}}: \frak{p}^{-d}\cal O_\frak{p}\{\{t\}\}\stackrel{\sim}{\rightarrow} (\frak{p}^{-d}\cal O_\frak{p}\{\{t\}\})^\perp$, hence $\gamma_{y_\frak{p}}(q)=1$ by Theorem \ref{thm5.1}. If ${\rm ord}_{y_\frak{p}}\omega=c_\psi+2d+1$, then
$(\frak{p}^{-d-1}\cal O_\frak{p}\{\{t\}\})^\perp\rho_{y_\frak{p}}^{-1}=\frak{p}^{-d}\cal O_\frak{p}\{\{t\}\}\subset \frak{p}^{-d-1}\cal O_\frak{p}\{\{t\}\}$ and $h_{y_\frak{p}}$ induces a quadratic character
\[
\bar{h}_{y_\frak{p}}(a)=\psi_\frak{p}(\frac{1}{2}\bar{c}_\frak{p}a^2)
\]
on the quotient
\[
\frak{p}^{-d-1}\cal O_\frak{p}\{\{t\}\}/ \frak{p}^{-d}\cal O_\frak{p}\{\{t\}\}\stackrel{\sim}{\rightarrow} k_\frak{p}((t)),\quad f+\frak{p}^{-d}\cal O_\frak{p}\{\{t\}\}\mapsto \varpi_\frak{p}^{d+1}f \textrm{ mod } y_\frak{p}.
\]
This proves that $\gamma_\frak{p}(\omega)$ defined in the corollary also coincides with the local Weil index $\gamma_{y_\frak{p}}(q)$. From the general theory of Weil index for LCA(2), we know a priori that $\gamma_P(\omega)$ and $\gamma_\frak{p}(\omega)$ are independent of the choice of local parameters $t_y$, $y=y_P$ or $y_\frak{p}$,  and the choice of $\varpi_\frak{p}$ as well. This can be also verified directly by showing that $c_{0,P}'/c_{0,P}$ and $\bar{c}_\frak{p}'/\bar{c}_\frak{p}$ are both squares, if $c_{0,P}'$ and $\bar{c}_\frak{p}'$ are the resulting elements upon choosing different local parameters $t_y'$ and $\varpi_\frak{p}'$.
\end{proof}

\end{document}